\documentclass[a4paper,12pt,reqno]{amsart}
\usepackage{amssymb,amsthm}
\usepackage{ifthen}
\usepackage{graphicx}
\usepackage{float}
\usepackage{comment}

\theoremstyle{plain}

\newtheorem{thm}{Theorem}[section]
\newtheorem*{thm*}{Theorem}

\numberwithin{equation}{section}

\newtheorem{cor}{Corollary}[section]
\newtheorem{lem}{Lemma}[section]
\newtheorem{prop}{Proposition}[section]
\newtheorem{rem}{Remark}[section]
\newtheorem{case}{Case}

\newtheorem{defn}{Definition}[section]

\theoremstyle{definition}
\newcounter {own}
\def\theown {\thesection  .\arabic{own}}

\newenvironment{pf}[1][]{%
 \vskip 3mm
 \noindent
 \ifthenelse{\equal{#1}{}}%
  {{\slshape Proof. }}%
  {{\slshape #1.} }%
 }%
{\qed\bigskip}

\newcounter{alphabet}
\newcounter{tmp}

\newcounter{minutes}
\divide\time by 60
\newcounter{hours}
\multiply\time by 60 \addtocounter{minutes}{-\time}

\begin{document}
\bibliographystyle{amsplain}
\title{Existence of an extremal of Dunkl-type Sobolev inequality and of Stein-Weiss inequality for D-Riesz potential }

\thanks{
File:~\jobname .tex,
          printed: \number\year-\number\month-\number\day,
          \thehours.\ifnum\theminutes<10{0}\fi\theminutes}

\author{Saswata Adhikari $^\dagger$}

\address{Saswata Adhikari, School of Mathematical Sciences,
NISER Bhubaneswar, Odisha-752050, India.}
\email{saswata.adhikari@gmail.com}
\author{V. P. Anoop }
\address{V. P. Anoop, School of Mathematical Sciences,
NISER Bhubaneswar, Odisha-752050, India.}
\email{anoop.vp@niser.ac.in}
\author{Sanjay Parui }
\address{Sanjay Parui, School of Mathematical Sciences,
NISER Bhubaneswar, Odisha-752050, India.}
\email{parui@niser.ac.in}
\keywords{Dunkl transform, Riesz potential, Sobolev inequality,  Stein Weiss inequality, .\\
$^\dagger$ {\tt Corresponding author}
}

\maketitle
\pagestyle{myheadings}
\markboth{Saswata Adhikari, Anoop V. P., Sanjay Parui,}{Existence of an extremal of Sobolev type inequalities associated with Dunkl gradient}

\begin{abstract}
In this paper, we prove the existence of an extremal for the Dunkl-type Sobolev inequality in case of $p=2$. Also we prove the existence of an extremal of the Stein-Weiss inequality for the D-Riesz potential in case of $r=2$. 
\end{abstract}

\section {Introduction}
The classical Sobolev inequality states that for all $u\in C_{c}^{\infty}(\mathbb{R}^{d})$, 
\begin{eqnarray}\label{5peq22}
\|u\|_{L^{q}(\mathbb{R}^{d})}\leq C \|\nabla u\|_{L^{p}(\mathbb{R}^{d})},
\end{eqnarray}
where $1\leq p<d$, $q=\frac{dp}{d-p}$ and the constant $C>0$ only depends on d. This inequality plays an important role in analysis and as such it has been studied by many for e.g. see (\cite{ehl,vg,ls}). The problem of finding sharp constant to inequality (\ref{5peq22}) was answered in \cite{gt} and therein the author found the existence of an extremal function for which the equality holds in (\ref{5peq22}) for the case $1<p<\infty$.

One may consider inequality (\ref{5peq22}) in the context of Dunkl setting, which we say Dunkl-type Sobolev inequality, by replacing the Euclidean gradient $\nabla u$ by Dunkl gradient $\nabla_{k}u$ and Lebesgue measure $dx$ by the weighted measure $w_{k}(x)dx$. Dunkl-type Sobolev inequality was obtained as a corollary of a result on Riesz transform in \cite {ba} (also see \cite {shs}). In fact, it was proved there that for $1<p<d_{k}$ with $q=\frac{d_{k}p}{d_{k}-p}$,  one has
\begin{eqnarray}\label{eqn7}
\|u\|_{L^{q}(\mathbb{R}^{d},w_{k})}\leq C_{p,q} \|\nabla_{k} u\|_{L^{p}(\mathbb{R}^{d},w_{k})},
\end{eqnarray}
for all $u\in\mathcal{S}(\mathbb{R}^{d})$.

The first aim of this paper is to prove the existence of extremals of the inequality (\ref{eqn7}) in the case of $p=2$. Towards this, for $u\in\dot{H}^{1}(\mathbb{R}^{d},w_{k})$, we consider the function
\begin{eqnarray}\label{5peq24}
F(u)=\frac{\int\limits_{\mathbb{R}^{d}}|\nabla_{k} u|^{2}w_{k}(x) dx}{\big(\int\limits_{\mathbb{R}^{d}}|u|^{q}w_{k}(x) dx\big)^{\frac{2}{q}}},
\end{eqnarray}
where $q=\frac{2d_{k}}{d_{k}-2}$ and $\dot{H}^{1}(\mathbb{R}^{d},w_{k})=\dot{W^{1,2}}(\mathbb{R}^{d},w_{k})$ with the norm $\|u\|_{\dot{H}^{1}(\mathbb{R}^{d},w_{k})}=\|\nabla_{k}u\|_{L^{2}(\mathbb{R}^{d},w_{k})}$. Our goal is to show that infimum is attained for the function $F$ when the infimum is taken over all non-vanishing functions $u\in\dot{H}^{1}(\mathbb{R}^{d},w_{k})$. Towards this we first prove a Dunkl-type refined Sobolev inequality. This type of refined Sobolev inequality on $\mathbb{R}^{d}$ is proved in more general setting in \cite{ml}. Recently, similar problem has been considered in \cite{av} by A. Velicu, wherein he shows that the function $F$ defined in (\ref{5peq24}) attains an infimum and have found the best constant. But our proof of the existence of a minimizer is different from that of \cite {av}. Our approach to this problem is mainly based on \cite{rl}.

In the Eucledean space $\mathbb{R}^{d}$, the negative powers of the Laplacian can be defined as an integral representation in terms of the Riesz potential or fractional integral operator as follows:
\begin{eqnarray*}
(-\Delta)^{-\frac{\alpha}{2}} f(x) =I_{\alpha}f(x)=(c_{\alpha})^{-1}\int\limits_{\mathbb{R}^{d}} f(y)|x-y|^{\alpha-d}dy,
\end{eqnarray*}
where $0<\alpha<d$ and $c_{\alpha}=2^{\alpha-\frac{d}{2}}\frac{\Gamma(\frac{\alpha}{2})}{\Gamma(\frac{d-\alpha}{2})}$.
One such fundamental result for the Riesz potential operator is the Stein-Weiss inequality which gives the weighted $(L^{r},L^{s})$ boundedness:
\begin{thm}\cite{es}\label{5pnth2}
Let $d\in \mathbb{N}, 1< r\leq s<\infty,\gamma >-\frac{d}{s},\beta\geq \gamma,0<\alpha<d, \beta<\frac{d}{r^{\prime}},\alpha+\gamma-\beta=d(\frac{1}{r}-\frac{1}{s})$. Then 
\begin{eqnarray}\label{5peqn14}
\||x|^{\gamma}I_{\alpha}f\|_{L^{s}(\mathbb{R}^{d})}\leq C \||x|^{\beta}f\|_{L^{r}(\mathbb{R}^{d})},~\forall~f\in L^{r}(\mathbb{R}^{d},|x|^{\beta r}).
\end{eqnarray}
\end{thm}
The study of above kind of integral inequalities are having great importance in harmonic analysis. In \cite {Lieb}, Lieb used the method of rearrangement technique and symmetrization to prove the existence of extremals of (\ref{5peqn14}) in some cases. Later in \cite{chen}, the authors extended Lieb's result on the Heisenberg group under certain assumptions. We refer the readers to the articles \cite{beck2,beck4,han1,han2} to understand more about Stein-Weiss inequalities and its extremal functions.

S. Thangavelu and Y. Xu  in \cite {sty} defined the D-Riesz potential operator on Schwartz spaces as follows:
\begin{eqnarray*}
I_{\alpha}^{k}f(x)=(c_{\alpha}^{k})^{-1}\int\limits_{\mathbb{R}^{d}}\tau_{y}^{k}f(x)|y|^{\alpha-d_{k}}w_{k}(y)dy,
\end{eqnarray*}
where $0<\alpha<d_{k}$ and $c_{\alpha}^{k}=2^{\alpha-\frac{d_{k}}{2}}\frac{\Gamma(\frac{\alpha}{2})}{\Gamma(\frac{d_{k}-\alpha}{2})}$. The D-Riesz potential operator $I_{\alpha}^{k}f$ can also be written as 
\begin{eqnarray}\label{eqn1}
I_{\alpha}^{k}f(x)= (c_{\alpha}^{k})^{-1}\int\limits_{\mathbb{R}^{d}} f(y)\Phi(x,y) w_{k}(y)dy,
\end{eqnarray}
where $\Phi(x,y)=\tau_{y}^{k}|.|^{\alpha-d_{k}}(x)$.

In \cite{dg}, D. V. Gorbachev et al. proved the following Stein-Weiss inequality for the D-Riesz potential operator.
\begin{thm}\label{5pth1}
Let $d\in \mathbb{N}, 1< r\leq s<\infty,\gamma >-\frac{d_{k}}{s},\beta\geq \gamma,0<\alpha<d_{k}, \beta<\frac{d_{k}}{r^{\prime}},\alpha+\gamma-\beta=d_{k}(\frac{1}{r}-\frac{1}{s})$. Then 
\begin{eqnarray}\label{5peq14}
\||x|^{\gamma}I_{\alpha}^{k}f\|_{L^{s}(\mathbb{R}^{d},w_{k})}\leq C_{k} \||x|^{\beta}f\|_{L^{r}(\mathbb{R}^{d},w_{k})}~\forall~f\in L^{r}(\mathbb{R}^{d},|x|^{\beta r}w_{k}).
\end{eqnarray}
\end{thm}
If $k\equiv 0$, then Theorem \ref{5pth1} becomes Theorem \ref{5pnth2}. When $\beta=\gamma=0$ and $G=\mathbb{Z}_2^d$, Theorem \ref{5pth1} is proved by S. Thangavelu and Y. Xu in \cite{sty} and for any reflection group $G$ with $\beta=\gamma=0$, it is proved in \cite{shs}. The weighted case was proved by C. Abdelkefi and M. Rachdi in \cite{cam} when $r=s$ under more restrictive assumptions. 

The second aim of this paper is to prove the existence of an extremal of the inequality (\ref{5peq14}) in case of $r=2$. By definition the best constant $W_{k}$ in (\ref{5peq14}) is given by 
\begin{eqnarray}\label{5peq27}
W_{k}=\sup\frac {\||x|^{\gamma}I_{\alpha}^{k}f\|_{L^{s}(\mathbb{R}^{d},w_{k})}}{\||x|^{\beta}f\|_{L^{2}(\mathbb{R}^{d},w_{k})}},
\end{eqnarray}
where the supremum is taken over all non-vanishing functions $f\in L^{2}(\mathbb{R}^{d},w_{k})$. We first obtain weighted norm inequalities for the Dunkl-type heat semigroup operator $e^{t\Delta_{k}}$ and an improved version of inequality (\ref{5peq14}) involving Besov norms of negative smoothness. Then for any compact set $\mathcal{K}\subset\mathbb{R}^{d}$, we prove the following compact embedding 
\begin{eqnarray*}
\dot{H}_{\beta,k}^{\alpha,r}(\mathbb{R}^{d})\subset L^{s}(\mathcal{K},|x|^{\gamma s}),
\end{eqnarray*}
where 
\begin{eqnarray}\label{5peq40}
\dot{H}_{\beta,k}^{\alpha,r}(\mathbb{R}^{d})=\{u=I_{\alpha}^{k}f:f\in L^{r}(\mathbb{R}^{d},|x|^{\beta r}w_{k})\}
\end{eqnarray}
is the weighted homogenous Sobolev space in the Dunkl setting , which is a Banach space with the norm $\|u\|_{\dot{H}_{\beta,k}^{\alpha,r}(\mathbb{R}^{d})}=\||x|^{\beta}f\|_{L^{r}(\mathbb{R}^{d},w_{k})}$. Using the above results, we prove that $W_{k}$ defined in (\ref{5peq27}) has a maximizer. We have adopted the techniques of \cite{pd} for this problem.

We organize the paper as follows. In section 2, we provide a brief introduction to Dunkl theory and some known results. In section 3, we prove a Dunkl-type refined Sobolev inequality. In section 4, we prove the existence of an extremal function for the Dunkl-type Sobolev inequality in case of $p=2$. In section 5, we provide weighted estimates for the operator $e^{t\Delta_{k}}$ in some cases. In section 6, we prove an improved version of Stein-Weiss inequality for D-Riesz potential operator. In section 7, we prove the existence of an extremal function of Stein-Weiss inequality for the D-Riesz potential operator in case of $r=2$.

\section{Preliminaries}
In this section, we shall briefly introduce the theory of Dunkl operators. For more details on Dunkl theory, we refer to \cite{cd,sy,rsl2}. 

For $\nu\in\mathbb{R}^{d}\setminus \{0\}$ let $\sigma_{\nu}$ denote the reflection of $\mathbb{R}^{d}$ in the hyperplane $\langle\nu\rangle^{\perp}$ given by the following formula:
\begin{eqnarray*}
\sigma_{\nu}(x)=x-2\frac{\langle\nu,x\rangle}{|\nu|^{2}}\nu.
\end{eqnarray*}
A finite subset $R$ of $\mathbb{R}^{d}\setminus \{0\}$ is said to be a root system if $R\cap \mathbb{R}_{\nu}=\{\nu,-\nu\}$ and $\sigma_{\nu}(R)=R,~\forall~\nu\in R$.
The set of reflections $\{\sigma_{\nu}:\nu\in R\}$ generates the subgroup $G:=G(R)$ of the orthogonal group $O(d,\mathbb{R})$, which is known as the reflection group associated with $R$. 
From now onwards let $R$ be a fixed root system in $\mathbb{R}^{d}$ and G be the associated reflection group . For simplicity, we assume $R$ to be normalized in the sense that $\langle\nu,\nu\rangle=2,~\forall~\nu\in R$.

A function $k: R\rightarrow\mathbb{C}$ is called a multiplicity function on the root system $R$ if it is invariant under the natural action of $G$ on $R$, that is, if $k(\sigma_{\nu}g)=k(g),~\nu,g\in R$. The set of all multiplicity functions forms a $\mathbb{C}$-vector space and it is denoted by K. 
\begin{defn}
Associated with $G$ and $k$, the Dunkl operator $T_{\xi}:=T(\xi)(k)$ is defined by (for $f\in C^{1}(\mathbb{R}^{d})$)
\begin{eqnarray*}
T_{\xi}f(x)=\partial_{\xi}f(x)+\sum\limits_{\nu\in R_{+}} k(\nu)\langle\nu,\xi\rangle\frac{f(x)-f(\sigma_{\nu}(x))}{\langle\nu,x\rangle},~~~\xi\in\mathbb{R}^{d},
\end{eqnarray*}
where $\partial_{\xi}$ denotes the directional derivative in the direction of $\xi$ and $R_{+}$ is a fixed positive subsystem of $R$.  
\end{defn}
For $\xi=e_{i}$, we shall write $T_{i}$ for $T_{e_{i}}$.
We denote Dunkl gradient by $\nabla_{k}=(T_{1},T_{2},\dotsc,T_{d})$ and Dunkl Laplacian by $\Delta_{k}=\sum\limits_{i=1}^{d}T_{i}^{2}$. Throughout the paper we assume that $k\geq 0$. Let $w_{k}$ denote the weight function defined by
\begin{eqnarray}\label{5peq6}
w_{k}(x)=\prod_{\nu\in R_{+}}|\langle\nu,x\rangle|^{2k(\nu)},~x\in\mathbb{R}^{d},
\end{eqnarray}
which is a $G$-invariant homogeneous function of degree $2\gamma_{k}$ with $\gamma_{k}=\sum\limits_{\nu\in R_{+}} k(\nu)$, that is, $w_{k}(cx)=|c|^{2\gamma_{k}} w_{k}(x),\forall~c\in\mathbb{R}$ . Let $d_{k}=d+2\gamma_{k}$. Further, we define the constants $c_{k}=\int\limits_{\mathbb{R}^{d}} e^{-\frac{|x|^{2}}{2}} w_{k}(x)dx $ and $a_{k}=\int\limits_{S^{d-1}}w_{k}(x^{\prime})dx^{\prime}$. Then $c_{k}$ and $a_{k}$ are related by the following formula
\begin{eqnarray}\label{5peq29}
c_{k}=2^{\frac{d_{k}}{2}-1}\Gamma \big(\frac{d_{k}}{2}\big)a_{k}.
\end{eqnarray}
There exists a unique linear isomorphism $V_{k}$ on polynomials, which intertwines the associated commutative algebra of Dunkl operators and the algebra of usual partial differential operators. Using the function $V_{k}$, one can define the Dunkl kernel $E_{k}$ as follows:
\begin{eqnarray*}
E_{k}(x,y):= V_{k}(e^{\langle .,y\rangle}) (x), ~x\in\mathbb{R}^{d}, y\in\mathbb{C}^{d}. 
\end{eqnarray*}
For $k\equiv 0$, the Dunkl kernel $E_{k}$ reduces to the usual expotential function $e^{ix.y}$. 
Alternatively, it is the solution of a joint eigen value problem for the associated Dunkl operators. We collect few properties of the Dunkl kernel $E_{k}$. 
\begin{prop} \label{5ppr1}
Let $k\geq 0, x,y\in\mathbb{C}^{d},\lambda\in\mathbb{C},\alpha\in\mathbb{Z}_{+}^{d}$. 
\begin{itemize}
    \item [(i)]  $E_{k} (x,y)= E_{k} (y,x)$
    \item [(ii)] $E_{k}(\lambda x,y)=E_{k}(x,\lambda y)$
    \item [(iii)] $\overline{E_{k}(x,y)}=E_{k}(\overline{x},\overline{y})$. 
    \item [(iv)] $|\partial_{y}^{\alpha} E_{k}(x,y)|\leq |x|^{|\alpha|}\max\limits_{g\in G} e^{Re\langle gx,y\rangle}$.
    
    In particular, $E_{k}(-ix,y)\leq 1$ and $|E_{k}(x,y)|\leq e^{|x||y|},~\forall~ x,y\in\mathbb{R}^{d}$. 
\end{itemize}
\end{prop}
Using the Dunkl kernel one can define Dunkl transform , which is the generalization of classical Fourier transform. Dunkl transform enjoins similar properties to that of classical Fourier transform.  
\begin{defn}\label{5pdefn1}
For a function $f\in L^{1}(\mathbb{R}^{d},w_{k})$, the Dunkl transform associated with $G$ and $k\geq 0$, denoted by $\mathcal{F}_{k}f$, is defined as 
\begin{eqnarray*}
\mathcal{F}_{K}(f)(\xi)= c_{k}^{-1}\int\limits_{\mathbb{R}^{d}}f(x)E_{k}(-i\xi,x)w_{k}(x)dx,~~\xi\in\mathbb{R}^{d}.
\end{eqnarray*}
\end{defn}
When $k\equiv 0$, the Dunkl transform reduces to the classical Fourier transform. The Dunkl transform can be extended to an isometric isomorphism between $L^{2}(\mathbb{R}^{d},w_{k})$ and $L^{2}(\mathbb{R}^{d},w_{k})$ i.e., for $f\in L^{2}(\mathbb{R}^{d},w_{k})$, one has
\begin{eqnarray}\label{5peq1}
\|f\|_{L^{2}(\mathbb{R}^{d}, w_{k})}^{2}=\|\mathcal{F}_{k}(f)\|_{L^{2}(\mathbb{R}^{d}, w_{k})}^{2}.
\end{eqnarray}
The usual translation operator $f\longmapsto f(.-y)$ leaves the Lebesgue measure on $\mathbb{R}^{d}$ invariant. However the measure $w_k(x)dx$ is no longer invariant under the usual translation and the Leibniz's formula $T_i(fg)=fT_ig+gT_if$ does not hold in general. So one can introduce the notion of a generalized translation operator defined on the Dunkl transform by the formula
\begin{eqnarray}\label{5peq38}
\mathcal{F}_{k}(\tau_{y}^{k}f)(\xi)= E_{k}(iy,\xi) \mathcal{F}_{k}(f)(\xi).
\end{eqnarray}
In case when $k\equiv 0$, $\tau_{y}^{k}f$ reduces to the usual translation $\tau_{y}^{0}f(x)=f(x+y)$. If $f,g\in L^{2}(\mathbb{R}^{d},w_{k})$, then
\begin{eqnarray}\label{eqn3}
\int\limits_{\mathbb{R}^{d}}\tau_{y}^{k}f(\xi)g(\xi)w_{k}(\xi)d\xi=\int\limits_{\mathbb{R}^{d}}f(\xi)\tau_{-y}^{k}g(\xi)w_{k}(\xi)d\xi,~\forall~y\in\mathbb{R}^{d}. 
\end{eqnarray}
In general, the explicit expression for $\tau_{y}^{k}f$ is unknown. It is known only when either $f$ is a radial function or $G=\mathbb{Z}_{2}^{d}$.  
The convolution of two functions $f,g\in L^{2}(\mathbb{R}^{d}, w_{k})$ is defined as follows:
\begin{eqnarray*}
(f*_{k}g)(x)=\int\limits_{\mathbb{R}^{d}}f(y)\tau_{y}^{k}g(x)w_{k}(y)dy.
\end{eqnarray*}
The convolution operator satisfies the following basic properties:
\begin{itemize}
\item [(i)] $\mathcal{F}_{k}(f*_{k}g)=\mathcal{F}_{k}(f).\mathcal{F}_{k}(g)$ 
\item [(ii)] $f*_{k}g=g*_{k}f$.
\end{itemize}
Using the convolution operator, the heat semi-group operator $e^{t\Delta_{k}}$ is defined as follows: $e^{t\Delta_{k}}u=u*_{k}q_{t}^{k}$, where 
\begin{eqnarray}\label{5peq2}
q_{t}^{k}(x)=(2t)^{-(\gamma_{k}+\frac{d}{2})}e^{-\frac{|x|^{2}}{4t}},~x\in\mathbb{R}^{d}.
\end{eqnarray}
In \cite{rsl}, it has been shown that the function $q_{t}^{k}(x)$ satisfies the Dunkl-type heat equation $\Delta_{k}u-\partial_{t}u=0$ on $\mathbb{R}^{d}\times (0,\infty)$. A short calculation using the properties of Dunkl transform shows that 
\begin{eqnarray}\label{5peq3}
\mathcal{F}_{k}(q_{t}^{k})(\xi)=e^{-t|\xi|^{2}},
\end{eqnarray}
and 
\begin{eqnarray}\label{5peq4}
\tau_{y}^{k}q_{t}^{k}(x)= (2t)^{-(\gamma_{k}+\frac{d}{2})}e^{-\frac{|x|^{2}+|y|^{2}}{4t}} E_{k}\bigg(\frac{x}{\sqrt{2t}},\frac{y}{\sqrt{2t}}\bigg).
\end{eqnarray}
From (\ref{5peq4}), it is observed that $\tau_{y}^{k}q_{t}^{k}(x)=\tau_{x}^{k}q_{t}^{k}(y)$.

We recall few results which will be useful in this paper. 
\begin{thm}\cite{dvg}
Let $1\leq p\leq\infty$ and $g$ is a Schwartz class radial function. Then for any $y\in\mathbb{R}^{d}$,
\begin{eqnarray}\label{5peq5}
\|\tau_{y}^{k}g\|_{L^{p}(\mathbb{R}^{d},w_{k})}\leq \|g\|_{L^{p}(\mathbb{R}^{d},w_{k})}. 
\end{eqnarray}
\end{thm}

\begin{lem}(Brezis Lieb Lemma)\label{5plem4}
Let $(X,dx)$ be a measure space and $(f_{j})$ be a bounded sequence in $L^{p}(X),~0<p<\infty$, which converges pointwise a.e. to a function $f$. Then
\begin{eqnarray*}
\lim_{j\rightarrow\infty}\int\limits_{X}||f_{j}|^{p}-|f_{j}-f|^{p}-|f|^{p}|dx=0.
\end{eqnarray*}
\end{lem}

\begin{thm}\cite {jpa}
In Theorem 4.1, J. P. Anker et al. obtained the following estimate. 

For any non-negative integer $m$ and for any multi-indices $\alpha,\beta$, there exists constant $C_{m,\alpha,\beta}>0$ such that for any $t>0$ and for any $x,y\in\mathbb{R}^{d}$, the following estimate holds:
\begin{eqnarray}\label{5peq54}
|\partial_{t}^{m}\partial_{x}^{\alpha}\partial_{y}^{\beta}h_{t}(x,y)|\leq C_{m,\alpha,\beta} t^{-m-\frac{|\alpha|}{2}-\frac{|\beta|}{2}} h_{2t}(x,y),
\end{eqnarray}
where $h_{t}(x,y)=\tau_{y}^{k}q_{t}^{k}(x)$.
\end{thm}

\begin{thm}\cite{rsl1}\label{5pthm10}
For a radial Schwartz class function $f(x)=f_{o}(|x|)$, one has 
\begin{eqnarray}\label{eqn2}
\tau_{y}^{k}f(x)=\int\limits_{\mathbb{R}^{d}} f_{0}(\sqrt{|x|^{2}+|y|^{2}-2\langle y,\eta\rangle}) d\mu_{x}^{k}(\eta),
\end{eqnarray}
where for each x in $\mathbb{R}^{d}$, $d\mu_{x}^{k}$ is a probability measure on $\mathbb{R}^{d}$, whose support is contained in $co(G.x)$, the convex hull of the $G$-orbit of $x$.
\end{thm}

\begin{lem}\cite {dg}\label{5plem6}
The kernel $\Phi(x,y)$ defined in (\ref{eqn1}) satisfies the following properties:
\begin{itemize}
    \item [(i)] $\Phi(x,y)=\Phi(y,x)$,
    \item [(ii)] $\Phi(rx,ty)=r^{\alpha-d_{k}}\Phi(x,\frac{ty}{r})$,
    \item [(iii)] $\Phi(x,y)=(c_{\alpha}^{k})^{-1}\int\limits_{\mathbb{R}^{d}}(|x|^{2}+|y|^{2}-2\langle y,\eta\rangle)^{\frac{\alpha-d_{k}}{2}}d\mu_{x}^{k}(\eta)$,
\end{itemize}
where the measure $d\mu_{x}^{k}$ is defined similarly as in Theorem \ref{5pthm10}.
\end{lem}

We observe the following properties of the weight function $w_{k}$. 

{\textbf{Properties of $w_{k}$:}}
\begin{enumerate}
\item [(1)] For $c>0$, $\int\limits_{\mathbb{R}^{d}} e^{-c|x|^{2}}w_{k}(x)dx= c^{-\frac{d_{k}}{2}}\int\limits_{\mathbb{R}^{d}} e^{-|x|^{2}} w_{k}(x)dx=c^{-\frac{d_{k}}{2}}\frac {a_{k}}{2}\Gamma \big(\frac{d_{k}}{2}\big)$.
\begin{pf}
By substituting $\sqrt{c}x=y$ and using (\ref{5peq6}) we get
\begin{eqnarray*}
\int\limits_{\mathbb{R}^{d}} e^{-c|x|^{2}}w_{k}(x)dx
= c^{-\frac{d}{2}}\int\limits_{\mathbb{R}^{d}} e^{-|y|^{2}}w_{k}\bigg(\frac{y}{\sqrt{c}}\bigg) dy
&=& c^{-\frac{d}{2}} \int\limits_{\mathbb{R}^{d}} e^{-|y|^{2}}\prod_{\nu\in R_{+}}\big|\langle\nu,\frac{y}{\sqrt{c}}\rangle\big|^{2k(\nu)}dy\\
&=&  c^{-\frac{d}{2}} \int\limits_{\mathbb{R}^{d}} e^{-|y|^{2}}\prod_{\nu\in R_{+}} c^{-k(\nu)}|\langle \nu, y\rangle|^{2k(\nu)}dy\\
&=& c^{-(\gamma_{k}+\frac{d}{2})}\int\limits_{\mathbb{R}^{d}} e^{-|y|^{2}} w_{k}(y)dy\\
&=& c^{-\frac{d_{k}}{2}}\int\limits_{\mathbb{R}^{d}} e^{-|x|^{2}} w_{k}(x)dx.
\end{eqnarray*}
Now by substituting $x=\frac{y}{\sqrt{2}}$ in the last integral and then using (\ref{5peq29}), we can write
\begin{eqnarray*}
\int\limits_{\mathbb{R}^{d}} e^{-|x|^{2}} w_{k}(x)dx=\frac {a_{k}}{2}\Gamma \big(\frac{d_{k}}{2}\big),
\end{eqnarray*}
thus proving property (1). 
\end{pf}
\item [(2)] If $R>0$ and $c<d_{k}$, then $\int\limits_{|y|\leq R}|x|^{-c}w_{k}(x)dx=\frac{a_{k}R^{d_{k}-c}}{d_{k}-c}$.
\begin{pf} Consider
\begin{eqnarray*}
\int\limits_{|y|\leq R}|x|^{-c}w_{k}(x)dx
&=&\int\limits_{0}^{R}\int\limits_{S^{d-1}} r^{-c}w_{k}(rx^{\prime})r^{n-1}dx^{\prime}dr\\
&=&\int\limits_{0}^{R}\int\limits_{S^{d-1}} r^{-c}r^{2\gamma_{k}}w_{k}(x^{\prime})r^{n-1}dx^{\prime}dr\\
&=& \int\limits_{0}^{R} r^{d_{k}-c-1}dr\int\limits_{S^{d-1}} w_{k}(x^{\prime}) dx^{\prime}=\frac{a_{k}R^{d_{k}-c}}{d_{k}-c},
\end{eqnarray*}
since the integrability condition at 0 is $c<d_{k}$, thus proving property (2). 
\end{pf}
\end{enumerate}

\section{Dunkl-type refined Sobolev inequality}
The goal of this section is to prove Dunkl-type refined Sobolev inequality (\ref{np8}). In order to prove this we first prove the following Pseudo-Poincare inequality in the Dunkl setting for $p=2$. 

\begin{lem}\label{5plem3}
For $u\in L^{2}(\mathbb{R}^{d},w_{k})$, one has
\begin{eqnarray*}
\|u-e^{t\Delta_{k}}u\|_{L^{2}(\mathbb{R}^{d},w_{k})}^{2}\leq t\|\nabla_{k}u\|_{L^{2}(\mathbb{R}^{d},w_{k})}^{2}.
\end{eqnarray*}
\end{lem}

\begin{pf} In order to prove the above Lemma, we shall make use of the following inequality.
\begin{eqnarray*}
(1-e^{-x})^{2}\leq 1-e^{-x}\leq x,~\forall~x\geq 0.
\end{eqnarray*}
Now, using the Plancherel formula(\ref{5peq1}) and (\ref{5peq3}), we get
\begin{eqnarray*}
\|u-e^{t\Delta_{k}}u\|_{L^{2}(\mathbb{R}^{d},w_{k})}^{2}
&=& \|\mathcal{F}_{k}(u-e^{t\Delta_{k}u})\|_{L^{2}(\mathbb{R}^{d},w_{k})}^{2}\\
&=&\|\mathcal{F}_{k}(u)-\mathcal{F}_{k}(u\ast_{k}q_{t}^{k})\|^{2}\\
&=& \|\mathcal{F}_{k}(u)-\mathcal{F}_{k}(u)\mathcal{F}_{k}(q_{t}^{k})\|^{2}\\
&=& \int\limits_{\mathbb{R}^{d}}|\mathcal{F}_{k}(u)(\xi)-\mathcal{F}_{k}(u)(\xi)\mathcal{F}_{k}(q_{t}^{k})(\xi)|^{2} w_{k}(\xi)d\xi\\
&=& \int\limits_{\mathbb{R}^{d}} |\mathcal{F}_{k}(u)(\xi)|^{2} (1-e^{-t|\xi|^{2}})^{2}w_{k}(\xi)d\xi\\
&\leq& t \int\limits_{\mathbb{R}^{d}} |\mathcal{F}_{k}(u)(\xi)|^{2}|\xi|^{2}w_{k}(\xi)d\xi\\
&=& t \|\mathcal{F}_{k}(\nabla_{k}u)\|_{L^{2}(\mathbb{R}^{d},w_{k})}^{2}
= t \|\nabla_{k}u\|_{L^{2}(\mathbb{R}^{d},w_{k})}^{2}.
\end{eqnarray*}
\end{pf}

\begin{thm}\label{5pth4}
For $d\geq 3$, there is a constant $C_{d,k}>0$ such that for all $u\in \dot{H}^{1}(\mathbb{R}^{d},w_{k})$, one has
\begin{eqnarray}\label{neq10}
&&\left(\int\limits_{\mathbb{R}^{d}}|u|^{q}(x)w_{k}(x)dx\right)^{\frac{1}{q}}\nonumber\\
&\leq& C_{d,k}\left(\int\limits_{\mathbb{R}^{d}}|\nabla_{k}u|^{2}(x)w_{k}(x)dx\right)^{\frac{1}{q}}\left(\sup\limits_{t>0}t^{\frac{(d_{k}-2)}{4}}\|e^{t\Delta_{k}}u\|_{\infty}\right)^{\frac{2}{d_{k}}},\label{np8}
\end{eqnarray}
with $q=\frac{2d_{k}}{d_{k}-2}$.
\end{thm}

\begin{pf}
Consider the function 
\begin{eqnarray*}
e^{t\Delta_{k}}u (x)=u*_{k}q_{t}^{k} (x)=\int\limits_{\mathbb{R}^{d}}u(y)(\tau_{y}^{k}q_{t}^{k})(x)w_{k}(y)dy. 
\end{eqnarray*}
Applying Holder's inequality to the function $e^{t\Delta_{k}}u$ with $p=\frac{2d_{k}}{d_{k}+2}$ and $q=\frac{2d_{k}}{d_{k}-2}$, we get
\begin{eqnarray}\label{5peq7}
|(e^{t\Delta_{k}}u) (x)|\leq \|u\|_{q,w_{k}}\|\tau_{x}^{k}q_{t}^{k}\|_{p,w_{k}}.
\end{eqnarray}
Now, using (\ref{5peq5}),
\begin{eqnarray*}
\|(\tau_{x}^{k}q_{t}^{k})\|_{p,w_{k}}^{p}
&=&\int\limits_{\mathbb{R}^{d}}|(\tau_{x}^{k}q_{t}^{k})(y)|^{p}w_{k}(y)dy\\
&\leq&\int\limits_{\mathbb{R}^{d}}|q_{t}^{k}(y)|^{p}w_{k}(y)dy.\\
\end{eqnarray*}
By substituting the value of $q_{t}^{k}$ from (\ref{5peq2}) in the last integral and using property (1) of section 2, we obtain
\begin{eqnarray}
\|(\tau_{x}^{k}q_{t}^{k})\|_{p,w_{k}}^{p}\nonumber
&\leq& \int\limits_{\mathbb{R}^{d}} |(2t)^{-(\gamma_{k}+\frac{d}{2})}e^{-\frac{|y|^{2}}{4t}}|^{p}w_{k}(y)dy\nonumber\\
&=& (2t)^{-(\gamma_{k}+\frac{d}{2})p}\int\limits_{\mathbb{R}^{d}} e^{-\frac{p}{4t}|y|^{2}}w_{k}(y)dy\nonumber\\
&=& (2t)^{-\frac{d_{k}}{2} p} \left(\frac{p}{4t}\right)^{-\frac{d_{k}}{2}}\frac {a_{k}}{2}\Gamma \big(\frac{d_{k}}{2}\big)\nonumber\\
&=& A_{d,k} t^{-\frac{d_{k}}{2}(p-1)},\label{5peq8}
\end{eqnarray}
where $A_{d,k}=2^{-\frac{d_{k}}{2} p}(\frac{p}{4})^{-\frac{d_{k}}{2}} \frac{a_{k}}{2}\Gamma\big(\frac{d_{k}}{2}\big)$ with $p=\frac{2d_{k}}{d_{k}+2}$. This implies that 
\begin{eqnarray*}
\|(\tau_{x}^{k}q_{t}^{k})\|_{p,w_{k}}
&\leq& A_{d,k}^{\frac{1}{p}} t^{-\frac{d_{k}}{2}(1-\frac{1}{p})}\\
&=& A_{d,k}^{\frac{1}{p}} t^{-\frac{d_{k}}{2}(1-\frac{d_{k}+2}{2d_{k}})}
=A_{d,k}^{\frac{1}{p}} t^{-\frac{d_{k}-2}{4}}.
\end{eqnarray*}
Then from (\ref{5peq7}), we get
\begin{eqnarray*}
\|e^{t\Delta_{k}}u\|_{\infty}\leq A_{d,k}^{\frac{1}{p}} \|u\|_{q,w_{k}} t^{-\frac{d_{k}-2}{4}}=C_{d,k} \|u\|_{q,w_{k}} t^{-\frac{d_{k}-2}{4}}.
\end{eqnarray*}
Let $I[u]=\sup\limits_{t>0}t^{\frac{(d_{k}-2)}{4}}\|e^{t\Delta_{k}}u\|_{\infty}$. Then $I[u]\leq C_{d,k}\|u\|_{q,w_{k}}$. Thus by homogeneity, we can assume that $I[u]\leq 1$ , that is, 
\begin{eqnarray}\label{5peq57}
t^{\frac{d_{k}-2}{4}} e^{t\Delta_{k}} u(x)\leq 1, ~\forall~t>0,~\forall~ x\in\mathbb{R}^{d},
\end{eqnarray}
and hence in order to prove (\ref{np8}), it is enough to show that 
\begin{eqnarray}\label{5peq23}
\int\limits_{\mathbb{R}^{d}}|u|^{q}(x)w_{k}(x)dx
\leq C_{d,k}^{q}\int\limits_{\mathbb{R}^{d}}|\nabla_{k}u|^{2}(x)w_{k}(x)dx.
\end{eqnarray}

Now we will be using some basic measure theory results in the proof. Recall that
\begin{eqnarray*}
|u(x)|^q=\int_{0}^{\infty}\chi_{\{|u(x)|^q>\lambda\}}d\lambda=q\int_0^\infty\chi_{\{|u(x)|>\tau\}}\tau^{q-1}d\tau.
\end{eqnarray*}
From this one can easily write that
\begin{eqnarray}\label{5peq58}
\int_{\mathbb{R}^d}|u(x)|^qw_k(x)dx=q\int_0^\infty |\{|u|>\tau\}|\tau^{q-1}d\tau
\end{eqnarray}
where $ |\{|u|>\tau\}|$ is the measure given by $ |\{|u|>\tau\}|=\int_{\mathbb{R}^d}\chi_{\{|u(x)|>\tau\}} w_k(x)dx$. If we write $u=(u-e^{t\Delta_{k}}u)+e^{t\Delta_{k}}u$ for some $t>0$ chosen later, then 
$$|\{|u|>\tau\}|\leq |\{|u-e^{t\Delta_k}u\}|>\tau/2|+|\{|e^{t\Delta_k}u|>\tau/2\}|.$$
Let us now choose $t=t_\tau$ satisfying  $\tau/2=t^{-\frac{d_k-2}{4}}$, then from (\ref{5peq57}), $|\{|e^{t_\tau \Delta_k}u|>\tau/2\}|=0$. Hence by (\ref{5peq58}) we have 
\begin{eqnarray}
\int_{\mathbb{R}^d}|u|^qw_k(x)dx \leq q\int_0^\infty |\{|u-e^{t_\tau\Delta_k}u|>\tau/2\}|\tau^{q-1}d\tau.
\end{eqnarray}

For a fixed constant $b\geq 1/16$ and for any $\tau>0$, we define a function $u_\tau$ on $\mathbb{R}^{d}$as follows:
\begin{eqnarray*}
u_\tau(x)=\begin{cases}
(b-\frac{1}{16})\tau &\quad\text{if } u(x)>b\tau,\\
u(x)-\frac{\tau}{16} &\quad\text{if } b\tau \geq u(x)\geq \frac{\tau}{16},\\
0 &\quad\text{if } \frac{\tau}{16} > u(x) >-\frac{\tau}{16},\\
u(x)+\frac{\tau}{16} &\quad\text{if }-\frac{\tau}{16} \geq u(x) \geq -b\tau,\\
-(b-\frac{1}{16})\tau &\quad\text{if } u(x) < -b\tau.
\end{cases}
\end{eqnarray*}
Note that $u_\tau$ is in $\dot{H}(\mathbb{R}^d,w_{k})$ and 
\begin{eqnarray*}
\int_{\mathbb{R}^d}|\nabla_ku_\tau|^2w_k(x)dx=\int_{\tau/16\leq|u|\leq b\tau}|\nabla_ku|^2w_k(x)dx.
\end{eqnarray*}
The decomposition $u-e^{t_\tau\Delta_k}u=(u_\tau-e^{t_\tau\Delta_k}u_\tau)-e^{t\tau\Delta_k}(u-u_\tau)+(u-u_\tau)$ gives

\begin{eqnarray}\label{5peq59}
|\{|u-e^{t_\tau\Delta_k}u|> \frac{\tau}{2}\}|
&\leq &
|\{u_\tau-e^{t_\tau\Delta_k}u_\tau|>\frac{\tau}{4}\}|
+|\{|u-u_\tau|>\frac{\tau}{8}\}|\nonumber \\
&& +|\{|e^{t\tau\Delta_k}(u-u_\tau)|>\frac{\tau}{8}\}|.
\end{eqnarray}
By using Chebyshev inequality with Lemma \ref{5plem3}, we get the bound for the first term of the right hand side of (\ref{5peq59})
\begin{eqnarray}
 |\{|u-e^{t_\tau\Delta_k}u|>\frac{\tau}{4}\}|\nonumber
&\leq& (\tau/4)^{-2}\|u_\tau-e^{t_\tau\Delta_k}u_\tau\|^{2}_{L^{2}(\mathbb{R}^{d},w_{k})}\nonumber\\
&\leq&  (\tau/4)^{-2} t_{\tau}\|\nabla_k u_\tau\|^{2}_{L^{2}(\mathbb{R}^{d},w_{k})}\nonumber\\
&\leq& 4(\tau/2)^{-q}\int_{\tau/16\leq|u|\leq b\tau}|\nabla_k u|^2w_k(x)dx\nonumber,
\end{eqnarray}
which implies that
\begin{eqnarray}\label{5peq60}
\int_0^\infty|\{|u_\tau -e^{-t_\tau \Delta_k}u_\tau|> \frac{\tau}{4} \}|\tau^{q-1}d\tau 
&=2^{q+2}log(16b)\int_{\mathbb{R}^d}|\nabla_k u|^2 w_k(x)dx.
\end{eqnarray}
Now we need to obtain the bound for the second and third term of the right hand side of (\ref{5peq59}). Towards this, first we observe that
\begin{eqnarray}\label{5peq66}
|u_\tau-u|=|u_\tau-u|\chi_{\{|u|\leq b\tau\}}+|u_\tau-u|\chi_{\{|u|>b \tau\}}\leq \frac{\tau}{16}+|u|\chi_{\{|u|>b \tau\}},
\end{eqnarray}
which leads to again by Chebyshev inequality
\begin{eqnarray}\label{5peq65}
|\{|u-u_\tau|>\frac{\tau}{8}\}|\leq |\{|u|\chi_{\{|u|>b \tau\}}>\frac{\tau}{16}\}|\leq (\tau/16)^{-1}\int_{\mathbb{R}^d}|u|\chi_{\{|u|>b \tau\}}w_k(x)dx.
\end{eqnarray}
Using the properties of Dunkl heat kernel and (\ref{5peq66}),
\begin{eqnarray}
|e^{t\Delta_k}u_\tau -e^{t\Delta_k} u| \leq e^{t\Delta_k}|u_\tau -u|\nonumber
&\leq& \frac{\tau}{16}\|\tau_{y}^{k}q_{t}^{k}\|_{L^{1}(\mathbb{R}^{d},w_{k})} + e^{t\Delta_k}(|u|\chi_{\{|u|>b\tau\}})\nonumber\\
&=& \frac{\tau}{16} c_{k}+e^{t\Delta_k}(|u|\chi_{\{|u|>b\tau\}}).\nonumber
\end{eqnarray}
Now we assume the $c_{k}\leq 1$. Then
\begin{eqnarray}
|e^{t\Delta_k}u_\tau -e^{t\Delta_k} u|\leq  \frac{\tau}{16}+ e^{t\Delta_k}(|u|\chi_{\{|u|>b\tau\}}).
\end{eqnarray}
Hence
\begin{eqnarray}
|\{|e^{t\Delta_k} (u_\tau - u)| > \frac{\tau}{8}\}| &\leq &|\{e^{t\Delta_k}(|u|\chi_{\{|u|>b\tau\}})>\frac{\tau}{16}\}|\nonumber \\&\leq & (\tau/16)^{-1}\int_{\mathbb{R}^d}e^{t\Delta_k}(|u|\chi_{\{|u|>b\tau\}})w_k(x)dx\nonumber\\
& =&(\tau/16)^{-1}c_{k}\int_{\mathbb{R}^d}|u|\chi_{\{|u|>b\tau\}}w_k(x)dx\nonumber\\
&\leq& (\tau/16)^{-1}\int_{\mathbb{R}^d}|u|\chi_{\{|u|>b\tau\}}w_k(x)dx\nonumber.
\end{eqnarray}
Then using (\ref{5peq65}), we have the estimate
\begin{eqnarray*}
&&\int_0^\infty (|\{|e^{t\Delta_k}u_\tau -e^{t\Delta_k} u| >\frac{\tau}{8}\}|+|\{|u-u_\tau|>\frac{\tau}{8}\}|)\tau^{q-1}d\tau \\
&=&\frac{32}{q-1}b^{-q+1}\int_{\mathbb{R}^d}|u|^qw_k(x)dx.
\end{eqnarray*}
Now from (\ref{5peq58}), using (\ref{5peq60}) and the above estimate we obtain for sufficiently large $b$,
\begin{eqnarray}
\int_{\mathbb{R}^d}|u|^q w_{k} (x)dx\leq \frac {q 2^{q+2}log(16b)}{1-\frac{32q}{q-1}b^{-q+1}}\int_{\mathbb{R}^d}|\nabla_k u|^2 w_{k}(x)dx.
\end{eqnarray}
thus proving (\ref{5peq23}).

Now let us assume that $c_{k}>1$. Choose $b>\frac{1}{16 c_{k}}$ and for any $\tau>0$, define the function $u_{\tau}$ on $\mathbb{R}^{d}$ as follows:
\begin{eqnarray*}
u_\tau(x)=\begin{cases}
(b-\frac {1}{16 c_{k}})\tau &\quad\text{if } u(x)>b\tau,\\
u(x)-\frac{\tau}{16 c_{k}} &\quad\text{if } b\tau \geq u(x)\geq \frac{\tau}{16 c_{k}},\\
0 &\quad\text{if } \frac{\tau}{16 c_{k}} > u(x) >-\frac{\tau}{16 c_{k}},\\
u(x)+\frac{\tau}{16 c_{k}} &\quad\text{if }-\frac{\tau}{16 c_{k}} \geq u(x) \geq -b\tau,\\
-(b-\frac{1}{16 c_{k}})\tau &\quad\text{if } u(x) < -b\tau.
\end{cases}
\end{eqnarray*}
Now proceeding as before we get,
\begin{eqnarray*}
\int_0^\infty|\{|u_\tau -e^{-t_\tau \Delta_k}u_\tau|> \frac{\tau}{4} \}|\tau^{q-1}d\tau 
&=2^{q+2}log(16b c_{k})\int_{\mathbb{R}^d}|\nabla_k u|^2 w_k(x)dx.
\end{eqnarray*}
Also, in this case
\begin{eqnarray*}
|u_\tau-u|\leq \frac{\tau}{16 c_{k}}+|u|\chi_{\{|u|>b \tau\}}\leq \frac{\tau}{16}+|u|\chi_{\{|u|>b \tau\}},
\end{eqnarray*}
and 
\begin{eqnarray*}
|e^{t\Delta_k}u_\tau -e^{t\Delta_k} u| \leq e^{t\Delta_k}|u_\tau -u|\nonumber
&\leq& \frac{\tau}{16c_{k}}\|\tau_{y}^{k}q_{t}^{k}\|_{L^{1}(\mathbb{R}^{d},w_{k})} + e^{t\Delta_k}(|u|\chi_{\{|u|>b\tau\}})\nonumber\\
&=& \frac{\tau}{16} +e^{t\Delta_k}(|u|\chi_{\{|u|>b\tau\}}).\nonumber
\end{eqnarray*}
Then proceeding exactly as before, for sufficiently large $b$, we have 
\begin{eqnarray*}
\int_{\mathbb{R}^d}|u|^q w_{k}(x) dx\leq \frac {q 2^{q+2}log(16b c_{k})}{1-\frac{32q c_{k}}{q-1}b^{-q+1}}\int_{\mathbb{R}^d}|\nabla_k u|^2 w_{k}(x)dx,
\end{eqnarray*}
thus proving (\ref{5peq23}). This completes the proof of Theorem \ref{5pth4}.
\end{pf}

\section{Existence of extremals for Dunkl-type Sobolev inequality}
The aim of this section is to prove the existence of a minimizer for the function $F$ defined in (\ref{5peq24}). In order to do so, first we prove the following corollary. 
\begin{cor}\label{5pcor1}
For $d\geq 3$, let $(u_{j})$ be a bounded sequence in $\dot{H}^{1}(\mathbb{R}^{d},w_{k})$. Then either one of the following statements holds. 
\begin{itemize}
    \item [(i)] $(u_{j})$ converges to 0 in $L^{q}(\mathbb{R}^{d},w_{k})$. 
    \item [(ii)] There exists a subsequence $(u_{jm})$ of $(u_{j})$ and sequences $(a_{m})\subset \mathbb{R}^{d}$ and $(b_{m})\subset (0,\infty)$ such that 
    \begin{eqnarray*}
    v_{m}(x)=b_{m}^{\frac{d_{k}-2}{2}}(\tau_{a_{m}}^{k}u_{jm})(b_{m}x)
    \end{eqnarray*}
\end{itemize}
converges weekly in $\dot{H}^{1}(\mathbb{R}^{d},w_{k})$ to a function $v\not\equiv 0$. Moreover, $(v_{m})$ converges pointwise a.e. to v. 
\end{cor}

\begin{pf}
Assume that $(i)$ does not hold. Then there exists $\epsilon>0$ and a subsequence $(u_{j,m})$ of $(u_{j})$ such that $\|u_{j,m}\|_{q,w_{k}}\geq\epsilon$. We shall denote $u_{j,m}$ by $u_{j}$ itself. Since $u_{j}$ is bounded in $\dot{H}^{1}(\mathbb{R}^{d},w_{k})$, there exists $A>0$ such that$\|\nabla_{k}u_{j}\|_{L^{2}(\mathbb{R}^{d},w_{k})}\leq \sqrt{A}~\forall ~j$. Now applying the Dunkl-type Sobolev inequality (\ref{neq10}) for the function $u_{j}$, we get 
\begin{eqnarray*}
\left(\sup\limits_{t>0}t^{\frac{(d_{k}-2)}{4}}\|e^{t\Delta_{k}}u_{j}\|_{\infty}\right)^{\frac{2}{d_{k}}}\geq C_{d,k}^{-1}A^{-\frac{d_{k}-2}{2d_{k}}}\epsilon.
\end{eqnarray*}

Then there exists $t_{j}>0$, $x_{j}\in\mathbb{R}^{d}$ such that 
\begin{eqnarray}\label{5peq36}
t_{j}^{\frac{d_{k}-2}{4}} |e^{t_{j}\Delta_{k}}u_{j}(x_{j})|\geq \frac{1}{2}C_{d,k}^{-\frac{d_{k}}{2}} A^{-\frac{d_{k}-2}{4}}\epsilon. 
\end{eqnarray}
Let $G(y)=2^{-\frac{d_{k}}{2}}e^{-\frac{|y|^{2}}{4}}$ and $v_{j}(y)=t_{j}^{\frac{d_{k}-2}{4}}(\tau_{-x_{j}}^{k}u_{j}) (\sqrt{t_{j}}y)$.
Now consider 
\begin{eqnarray}
\left|\int\limits_{\mathbb{R}^{d}}G(y)v_{j}(y)w_{k}(y)dy\right|
&=&\left|\int\limits_{\mathbb{R}^{d}}2^{-\frac{d_{k}}{2}}e^{-\frac{|y|^{2}}{4}}t_{j}^{\frac{d_{k}-2}{4}}(\tau_{-x_{j}}^{k}u_{j})(\sqrt{t_{j}}y)w_{k}(y)dy\right|\nonumber\\
&=& 2^{-\frac{d_{k}}{2}}t_{j}^{\frac{d_{k}-2}{4}}\left|\int\limits_{\mathbb{R}^{d}} e^{-\frac{|y|^{2}}{4t_{j}}}(\tau_{-x_{j}}^{k}u_{j})(y)w_{k}\left(\frac{y}{\sqrt{t_{j}}}\right)\frac{1}{t_{j}^\frac{d}{2}}dy\right|\nonumber\\
&=& 2^{-\frac{d_{k}}{2}}t_{j}^{\frac{d_{k}-2}{4}}t_{j}^{-\frac{d_{k}}{2}}\left|\int\limits_{\mathbb{R}^{d}}
e^{-\frac{|y|^{2}}{4t_{j}}}(\tau_{-x_{j}}^{k}u_{j})(y)w_{k}(y)dy\right|\nonumber.
\end{eqnarray}
Then using (\ref{eqn3}) and (\ref{5peq36})
\begin{eqnarray}
\left|\int\limits_{\mathbb{R}^{d}}G(y)v_{j}(y)w_{k}(y)dy\right|
&=& t_{j}^{\frac{d_{k}-2}{4}}\left|\int\limits_{\mathbb{R}^{d}}q_{t_{j}}(y)(\tau_{-x_{j}}^{k}u_{j})(y)w_{k}(y)dy\right|\nonumber\\
&=& t_{j}^{\frac{d_{k}-2}{4}}\left|\int\limits_{\mathbb{R}^{d}} u_{j}(y)(\tau_{x_{j}}^{k}q_{t_{j}})(y)w_{k}(y)dy\right|\nonumber\\
&=&  t_{j}^{\frac{d_{k}-2}{4}}\left|\int\limits_{\mathbb{R}^{d}} u_{j}(y)(\tau_{y}^{k}q_{t_{j}})(x_{j})w_{k}(y)dy\right|\nonumber\\
&=& t_{j}^{\frac{d_{k}-2}{4}} \left|e^{t_{j}\Delta_{k}}u_{j}(x_{j})\right|\nonumber\\
&\geq& \frac{1}{2}C_{d,k}^{-\frac{d_{k}}{2}} A^{-\frac{d_{k}-2}{4}}\epsilon.\label{np10}
\end{eqnarray}

Moreover,
\begin{eqnarray}\label{5peq37}
\mathcal{F}_{k}(v_{j})(\xi)=t_{j}^{-\frac{d_{k}+2}{4}}\mathcal{F}_{k}(\tau_{-x_{j}}^{k}u_{j})\big(\frac{\xi}{\sqrt{t_{j}}}\big).
\end{eqnarray}
Indeed, by inserting the value of $v_{j}$ in Definition \ref{5pdefn1}, we get
\begin{eqnarray*}
\mathcal{F}_{k}(v_{j})(\xi)
&=& c_{k}^{-1}\int\limits_{\mathbb{R}^{d}} v_{j}(y)E_{k}(-i\xi,y)w_{k}(y)dy\\
&=& c_{k}^{-1}\int\limits_{\mathbb{R}^{d}} t_{j}^{\frac{d_{k}-2}{4}}(\tau_{-x_{j}}^{k}u_{j})(\sqrt{t_{j}}y)E_{k}(-i\xi,y)w_{k}(y)dy.
\end{eqnarray*}
Now by replacing $y$ by $\frac{y}{\sqrt{t_{j}}}$ and proceeding as before, we obtain
\begin{eqnarray*}
\mathcal{F}_{k}(v_{j})(\xi)
&=& t_{j}^{\frac{d_{k}-2}{4}} t_{j}^{-\frac{d_{k}}{2}} c_{k}^{-1}\int\limits_{\mathbb{R}^{d}} (\tau_{-x_{j}}^{k}u_{j})(y)E_{k}\big(-i\xi,\frac{y}{\sqrt{t_{j}}}\big)w_{k}(y)dy\\
&=&  t_{j}^{-\frac{d_{k}+2}{4}} c_{k}^{-1}\int\limits_{\mathbb{R}^{d}} (\tau_{-x_{j}}^{k}u_{j})(y)E_{k}\big(-\frac{i\xi}{\sqrt{t_{j}}},y\big)w_{k}(y)dy\\
&=& t_{j}^{-\frac{d_{k}+2}{4}}\mathcal{F}_{k}(\tau_{-x_{j}}^{k}u_{j})\bigg(\frac{\xi}{\sqrt{t_{j}}}\bigg),
\end{eqnarray*}
using Proposition \ref{5ppr1}(ii) and the definition of Dunkl transform, thus proving (\ref{5peq37}). 
Therefore, using (\ref{5peq37}) 
\begin{eqnarray*}
\|\nabla_{k}v_{j}\|_{L^{2}(\mathbb{R}^{d},w_{k})}^{2}
&=& \|\mathcal{F}_{k}(\nabla_{k}v_{j})\|_{L^{2}(\mathbb{R}^{d},w_{k})}^{2}\\
&=& \int\limits_{\mathbb{R}^{d}} |\xi|^{2}|\mathcal{F}_{k}(v_{j})(\xi)|^{2}w_{k}(\xi)d\xi\\
&=& t_{j}^{-\frac{d_{k}+2}{2}}\int\limits_{\mathbb{R}^{d}}|\xi|^{2}\bigg|\mathcal{F}_{k}(\tau_{-x_{j}}^{k}u_{j})\bigg(\frac{\xi}{\sqrt{t_{j}}}\bigg)\bigg|^{2}w_{k}(\xi)d\xi\\
&=& t_{j}^{-\frac{d_{k}+2}{2}} t_{j}^{1+\gamma_{k}+\frac{d}{2}}  \int\limits_{\mathbb{R}^{d}} |\xi|^{2}|\mathcal{F}_{k}(\tau_{-x_{j}}^{k}u_{j})(\xi)^{2} w_{k}(\xi)d\xi.
\end{eqnarray*}
Consequently, using (\ref{5peq38}) and Proposition \ref{5ppr1} (iv), we get
\begin{eqnarray*}
\|\nabla_{k}v_{j}\|_{L^{2}(\mathbb{R}^{d},w_{k})}^{2}
&=& \int\limits_{\mathbb{R}^{d}} |\xi|^{2}|E_{k}(-ix_{j},\xi)\mathcal{F}_{k}(u_{j})(\xi)|^{2} w_{k}(\xi) d\xi\\
&\leq& \int\limits_{\mathbb{R}^{d}} |\xi|^{2}|\mathcal{F}_{k}(u_{j})(\xi)|^{2} w_{k}(\xi)d\xi
= \|\nabla_{k}u_{j}\|_{L^{2}(\mathbb{R}^{d},w_{k})}^{2}.
\end{eqnarray*}

Thus $\|\nabla_{k}v_{j}\|_{L^{2}(\mathbb{R}^{d},w_{k})}^{2}\leq A$ for all $j$. By Banach-Alaoglu theorem, $v_{j}$ has a weekly convergent subsequence in $\dot{H}^{1}(\mathbb{R}^{d},w_{k})$ and let it converge to $w$. Since $G\in \dot{H}^{1}(\mathbb{R}^{d},w_{k})^{*}$, $G(v_{j})$ converges to $G(w)$. It follows from (\ref{np10}) that $G(v_{j})\neq 0,~\forall~j $ and therefore, $G(w)\neq 0$, which in turn imply that $w\not\equiv 0$. This completes the proof of the corollary. 
\end{pf}

Now by using the above Corollary, we have the following theorem
\begin{thm}
Let $d\geq 3$. Then the infimum
\begin{eqnarray*}
S_{d,k}=\inf\limits_{u\in\dot{H}^{1}(\mathbb{R}^{d},w_{k})}\frac{\int\limits_{\mathbb{R}^{d}}|\nabla_{k}u|^{2}w_{k}(x)dx}{\left(\int\limits_{\mathbb{R}^{d}}|u|^{q}w_{k}dx\right)^{\frac{2}{q}}}
\end{eqnarray*}
is attained.
\end{thm}

\begin{pf}
Let $(u_{j})$ be a minimizing sequence, which we assume to be normalized in $L^{q}(\mathbb{R}^{d},w_{k})$. Then $(u_{j})$ is bounded in $\dot{H}^{1}({\mathbb{R}^{d}},w_{k})$, since the sequence $\|\nabla_{k}u_{j}\|_{L^{2}(\mathbb{R}^{d},w_{k})}^{2}$ converges to $S_{d,k}$. Moreover, because $(u_{j})$ has norm 1 in $L^{q}(\mathbb{R}^{d},w_{k})$, from Corollary \ref{5pcor1}, we can say that after a generalized translation and a dilation , $(u_{j})$ converges weekly to a non-zero function $u_{0}$ a.e. in $\dot{H}^{1}({\mathbb{R}^{d}},w_{k})$. This implies that $\langle u_{j},u_{0}\rangle_{\dot{H}^{1}({\mathbb{R}^{d}},w_{k})}$ converges to $\|u_{0}\|_{\dot{H}^{1}({\mathbb{R}^{d}},w_{k})}^{2}$, from which it follows that
\begin{eqnarray}
\lim_{j\rightarrow\infty}\int\limits_{\mathbb{R}^{d}}|\nabla_{k}(u_{j}-u_{0})|^{2}w_{k}(x)dx
&=&\lim_{j\rightarrow\infty}\int\limits_{\mathbb{R}^{d}}|\nabla_{k}u_{j}|^{2}w_{k}(x)dx-\int\limits_{\mathbb{R}^{d}}|\nabla_{k}u_{0}|^{2}w_{k}(x)dx\nonumber\\
&=& S_{d,k}-\int\limits_{\mathbb{R}^{d}}|\nabla_{k}u_{0}|^{2}w_{k}(x)dx.\label{eqn5}
\end{eqnarray}
Now from Lemma \ref{5plem4}, we have 
\begin{eqnarray*}
1=\lim_{j\rightarrow\infty}\int\limits_{\mathbb{R}^{d}}|u_{j}|^{q}w_{k}(x)dx=\lim_{j\rightarrow\infty}\int\limits_{\mathbb{R}^{d}}|u_{j}-u_{0}|^{q}w_{k}(x)dx+\int\limits_{\mathbb{R}^{d}}|u_{0}|^{q}w_{k}(x)dx.
\end{eqnarray*}
As a consequence, since $\frac{2}{q}<1$,
\begin{eqnarray}\label{eqn6}
1\leq \lim_{j\rightarrow\infty}\left(\int\limits_{\mathbb{R}^{d}}|u_{j}-u_{0}|^{q}w_{k}(x)dx\right)^{\frac{2}{q}}+\left(\int\limits_{\mathbb{R}^{d}}|u_{0}|^{q}w_{k}(x)dx\right)^{\frac{2}{q}},
\end{eqnarray}
using the fact that $(a+b)^{r}\leq a^{r}+b^{r}$ for $a,b>0, 0\leq r\leq 1$.

Hence using (\ref{eqn5}) and (\ref{eqn6}), we have
\begin{eqnarray*}
S_{d,k} &\geq& \frac{\lim_{j\rightarrow\infty}\int\limits_{\mathbb{R}^{d}}|\nabla_{k}(u_{j}-u_{0})|^{2}w_{k}(x)dx+\int\limits_{\mathbb{R}^{d}}|\nabla_{k}u_{0}|^{2}w_{k}(x)dx}{\lim_{j\rightarrow\infty}\left(\int\limits_{\mathbb{R}^{d}}|u_{j}-u_{0}|^{q}w_{k}(x)dx\right)^{\frac{2}{q}}+\left(\int\limits_{\mathbb{R}^{d}}|u_{0}|^{q}w_{k}(x)dx\right)^{\frac{2}{q}}}\\
&\geq& \frac{S_{d,k}\lim_{j\rightarrow\infty}\left(\int\limits_{\mathbb{R}^{d}}|u_{j}-u_{0}|^{q}w_{k}(x)dx\right)^{\frac{2}{q}}+\int\limits_{\mathbb{R}^{d}}|\nabla_{k}u_{0}|^{2}w_{k}(x)dx}{\lim_{j\rightarrow\infty}\left(\int\limits_{\mathbb{R}^{d}}|u_{j}-u_{0}|^{q}w_{k}(x)dx\right)^{\frac{2}{q}}+\left(\int\limits_{\mathbb{R}^{d}}|u_{0}|^{q}w_{k}(x)dx\right)^{\frac{2}{q}}}.
\end{eqnarray*}

This implies that 
\begin{eqnarray*}
S_{d,k}\geq \frac{\int\limits_{\mathbb{R}^{d}}|\nabla_{k}u_{0}|^{2}w_{k}(x)dx}{\left(\int\limits_{\mathbb{R}^{d}}|u_{0}|^{q}w_{k}(x)dx\right)^{\frac{2}{q}}},
\end{eqnarray*}
from which it follows that $u_{0}$ is a minimizer. 
\end{pf}

\section{ weighted estimates for the heat semi group operator $e^{t\Delta_{k}}$}
In this section, we prove the following Proposition involving weighted estimates for the operator $e^{t\Delta_{k}}$, using which we shall show that $e^{t\Delta_{k}}$ is a compact operator. For the operator $e^{t\Delta}$, such types of estimates are found in \cite{rlu}. 

\begin{prop}\label{5ppr2}
Let $d\geq 2, 1<r<\infty$. Assume that 
\begin{eqnarray*}
0<\beta<\frac{d_{k}}{r^{\prime}}
\end{eqnarray*}
and fix $t>0$. Then 
\begin{itemize}
    \item [(i)] $\|e^{t\Delta_{k}}f\|_{L^{\infty}(\mathbb{R}^{d})}\leq C_{d,r,\beta,t,k}\||x|^{\beta}f\|_{r,w_{k}}$
    \item [(ii)] $\||x|^{w}e^{t\Delta_{k}}f\|_{L^{\infty}(\mathbb{R}^{d})}\leq D_{d,r,\beta,t,k}\||x|^{\beta}f\|_{r,w_{k}}$ for $\beta\geq w>0$.
    \item [(iii)] $\|\partial_{x_{i}}e^{t\Delta_{k}}f\|_{L^{\infty}(\mathbb{R}^{d})}\leq E_{d,r,\beta,t,k}\||x|^{\beta}f\|_{r,w_{k}}, ~~i=1,2,\dotsc,{n}$.
\end{itemize}
\end{prop}

\begin{pf}
Consider
\begin{eqnarray}
|e^{t\Delta_{k}}f(x)|\nonumber
&\leq& \int\limits_{\mathbb{R}^{d}}|f(y)||\tau_{y}^{k}q_{t}^{k}(x)|w_{k}(y)dy\nonumber\\
&\leq& \bigg(\int\limits_{\mathbb{R}^{d}}|f(y)|^{r}|y|^{\beta r}w_{k}(y)dy\bigg)^{\frac{1}{r}} \bigg(\int\limits_{\mathbb{R}^{d}}|\tau_{y}^{k}q_{t}^{k}(x)|^{r^{\prime}}|y|^{-\beta r^{\prime}}w_{k}(y)dy\bigg)^{\frac{1}{r^{\prime}}}\nonumber\\
&=& \||y|^{\beta}f\|_{r,w_{k}} (I_{1}(x)+I_{2}(x)^\frac{1}{r^{\prime}},\label{5peq11}
\end{eqnarray}
where 
\begin{eqnarray*}
I_{1}(x)=\int\limits_{|y|\leq \sqrt {t}}|\tau_{y}^{k}q_{t}^{k}(x)|^{r^{\prime}}|y|^{-\beta r^{\prime}}w_{k}(y)dy
\end{eqnarray*}
and 
\begin{eqnarray*}
I_{2}(x)=\int\limits_{|y|> \sqrt {t}}|\tau_{y}^{k}q_{t}^{k}(x)|^{r^{\prime}}|y|^{-\beta r^{\prime}}w_{k}(y)dy.
\end{eqnarray*}
Substituting the value of $\tau_{y}^{k}q_{t}^{k}(x)$ from (\ref{5peq4}) and then using Proposition \ref{5ppr1}, 
\begin{eqnarray}
|\tau_{y}^{k}q_{t}^{k}(x)|
&=& \left|(2t)^{-(\gamma_{k}+\frac{d}{2})}e^{-\frac{|x|^{2}+|y|^{2}}{4t}}E_{k}\big(\frac{x}{\sqrt{2t}},\frac{y}{\sqrt{2t}}\big)\right|\nonumber\\
&\leq&  (2t)^{-(\gamma_{k}+\frac{d}{2})} e^{-\frac{|x|^{2}+|y|^{2}}{4t}} e^{\frac{|x||y|}{2t}}\nonumber\\
&=&  (2t)^{-\frac{d_{k}}{2}} e^{-\frac{(|x|-|y|)^{2}}{4t}}\label{5peq28}\\
&\leq& C t^{-\frac{d_{k}}{2}}, \forall~ x,y\in \mathbb{R}^{d},\nonumber
\end{eqnarray}
where $C=2^{-\frac{d_{k}}{2}}$. 
Then
\begin{eqnarray*}
I_{1}(x)\leq C^{r^{\prime}} t^{-\frac{d_{k}}{2}r^{\prime}}\int\limits_{|y|\leq \sqrt {t}}|y|^{-\beta r^{\prime}}w_{k}(y)dy.
\end{eqnarray*}
By property (2) of section 2, the above integral is finite if $\beta< \frac{d_{k}}{r^{\prime}}$ and equals to $\frac{a_{k}(\sqrt {t})^{{d_{k}-\beta r^{\prime}}}} {d_{k}-\beta r^{\prime}}$. Thus $I_{1}(x)\leq A_{d,r,\beta,k} t^{-\frac{d_{k}}{2}r^{\prime}} t^{\frac{1}{2}(d_{k}-\beta r^{\prime})}$, where $A_{d,r,\beta,k}=\frac{a_{k}C^{r^{\prime}}}{d_{k}-\beta r^{\prime}}$. 

On the other hand, over the integral $I_{2}$, since $|y|>\sqrt {t}$,  $|y|^{-\beta r^{\prime}}\leq t^{-\frac {\beta r^{\prime}}{2}}$ and hence
\begin{eqnarray}
I_{2}(x)\leq t^{-\frac {\beta r^{\prime}}{2}}\int\limits_{|y|>\sqrt{t}}|\tau_{y}^{k}q_{t}^{k}(x)|^{r^{\prime}}w_{k}(y)dy\nonumber
&\leq& t^{-\frac {\beta r^{\prime}}{2}}\int\limits_{\mathbb{R}^{d}}|\tau_{y}^{k}q_{t}^{k}(x)|^{r^{\prime}}w_{k}(y)dy\\\nonumber
&\leq& t^{-\frac {\beta r^{\prime}}{2}}\int\limits_{\mathbb{R}^{d}} |q_{t}^{k}(y)|^{r^{\prime}} w_{k}(y)dy\\\label{5peq32}
&\leq& B_{d,r,k} t^{-\frac {\beta r^{\prime}}{2}} t^{-\frac{d_{k}}{2}(r^{\prime}-1)},
\end{eqnarray}
by taking $p=r^{\prime}$ in (\ref{5peq8}), where $B_{d,r,k}=2^{-\frac{d_{k}}{2}r^{\prime}}\big(\frac{r^{\prime}}{4}\big)^{-\frac{d_{k}}{2}}\frac{a_{k}}{2}\Gamma\big(\frac{d_{k}}{2}\big)$.
Finally from (\ref{5peq11}) using the estimates of $I_{1}$ , $I_{2}$ and the fact that $(a+b)^{u}\leq a^{u}+b^{u}$ for $a,b>0, u<1$, we obtain 
\begin{eqnarray}
|e^{t\Delta_{k}}f(x)|
&\leq &  \||y|^{\beta}f\|_{L^{r}(\mathbb{R}^{d},w_{k})}\left(A_{d,r,\beta,k} t^{-\frac{d_{k}}{2}r^{\prime}}t^{\frac{1}{2}(d_{k}-\beta r^{\prime})} + B_{d,r,k} t^{-\frac {\beta r^{\prime}}{2}} t^{-\frac{d_{k}}{2}(r^{\prime}-1)}\right)^\frac{1}{r^{\prime}}\nonumber\\
&\leq & C_{d,r,\beta,k} \||y|^{\beta} f\|_{L^{r}(\mathbb{R}^{d},w_{k})} \left(t^{-\frac{d_{k}}{2}}t^{\frac{1}{2r^{\prime}}(d_{k}-\beta r^{\prime})} + t^{-\frac{\beta}{2}} t^{-\frac{d_{k}}{2r}}\right)\label{5peqn44}\\
&\leq& C_{d,r,\beta,t,k} \||y|^{\beta} f\|_{L^{r}(\mathbb{R}^{d},w_{k})}\nonumber,
\end{eqnarray}
where $C_{d,r,\beta,k}=(\max\{{A_{d,r,\beta,k},B_{d,r,k}}\})^{\frac{1}{r^{\prime}}}$ and\\ $C_{d,r,\beta,t,k}=C_{d,r,\beta,k} \left(t^{-\frac{d_{k}}{2}}t^{\frac{1}{2r^{\prime}}(d_{k}-\beta r^{\prime})} + t^{-\frac{\beta}{2}} t^{-\frac{d_{k}}{2r}}\right)$. 
Thus we have proved that 
\begin{eqnarray}\label{5peq31}
\|e^{t\Delta_{k}}f\|_{L^{\infty}(\mathbb{R}^{d})}\leq C_{d,r,\beta,t,k}\||x|^{\beta}f\|_{L^{r}(\mathbb{R}^{d},w_{k})},
\end{eqnarray}
thus proving (i).

In particular, when $r=2$, then from (\ref{5peqn44}), we get
\begin{eqnarray}\label{5peq46}
\|e^{t\Delta_{k}}f\|_{L^{\infty}(\mathbb{R}^{d})}\leq C_{d,\beta,k} t^{-\frac{1}{2}(\frac{d_{k}}{2}+\beta)}\||x|^{\beta}f\|_{L^{2}(\mathbb{R}^{d},w_{k})},
\end{eqnarray}
where $C_{d,\beta,k}=2 C_{d,2,\beta,k}$.

Now we shall prove (ii). since $w>0$, using (\ref{5peq31}) we observe that, for $|x|\leq 1$,  
\begin{eqnarray}\label{5peq34}
||x|^{w} e^{t\Delta_{k}} f(x)|\leq |e^{t\Delta_{k}} f(x)|\leq  C_{d,r,\beta,t,k} \||x|^{\beta}f\|_{L^{r}(\mathbb{R}^{d},w_{k})}.
\end{eqnarray}
So we assume that $|x|>1$. Consider
\begin{eqnarray*}
&&||x|^{w}e^{t\Delta_{k}}f(x)|\\
&\leq& |x|^{w}\int\limits_{\mathbb{R}^{d}} |f(y)| |\tau^{k}_{y}q_{t}^{k}(x)| w_{k}(y)dy\\
&\leq& \left(\int\limits_{\mathbb{R}^{d}}|f(y)|^{r}|y|^{\beta r} w_{k}(y)dy\right)^{\frac{1}{r}} \left(\int\limits_{\mathbb{R}^{d}}|x|^{w r^{\prime}}|y|^{-\beta r^{\prime}}|\tau_{y}^{k}q_{t}^{k}(x)|^{r^{\prime}}w_{k}(y)dy\right)^{\frac{1}{r^{\prime}}}\\
&\leq & \||y|^{\beta}f\|_{r,w_{k}}\left(\int\limits_{\mathbb{R}^{d}}|x|^{\beta r^{\prime}}|y|^{-\beta r^{\prime}}|\tau_{y}^{k}q_{t}^{k}(x)|^{r^{\prime}}w_{k}(y)dy\right)^{\frac{1}{r^{\prime}}},
\end{eqnarray*}
since  $|x|>1$ and $\beta\geq w$, it implies that $|x|^{(w-\beta)r^{\prime}}\leq 1$. Then 
\begin{eqnarray}\label{5peq33}
||x|^{w}e^{t\Delta_{k}}f(x)|\leq  \||y|^{\beta}f\|_{r,w_{k}} (I_{1}(x)+I_{2}(x))^{\frac{1}{r^{\prime}}},
\end{eqnarray}
where 
$I_{1} (x)=\int\limits_{|y|\leq \frac{|x|}{2}}|x|^{\beta r^{\prime}}|y|^{-\beta r^{\prime}}|\tau_{y}^{k}q_{t}^{k}(x)|^{r^{\prime}}w_{k}(y)dy$ and \\
$I_{2} (x)=\int\limits_{|y|\geq \frac{|x|}{2}}|x|^{\beta r^{\prime}}|y|^{-\beta r^{\prime}}|\tau_{y}^{k}q_{t}^{k}(x)|^{r^{\prime}}w_{k}(y)dy $. Over the first integral $I_{1} (x)$, $(|x|-|y|)^{2}\geq\frac {|x|^{2}}{4}~\forall ~y$, since $|y|\leq \frac{|x|}{2}$.  Consequently, from (\ref{5peq28}) we arrive at the bound
$\tau_{y}^{k}q_{t}^{k}(x)\leq (2t)^{-\frac{d_{k}}{2}}e^{-\frac{|x|^{2}}{16t}}$. Then

\begin{eqnarray*}
I_{1}(x)\leq (2t)^{-\frac {d_{k}}{2} r^{\prime}}\frac {|x|^{\beta r^{\prime}}}{e^{\frac{r^{\prime}}{16t}|x|^{2}}}\int\limits_{|y|<\frac{|x|}{2}} |y|^{-\beta r^{\prime}} w_{k}(y)dy. 
\end{eqnarray*}
For $\beta<\frac{d_{k}}{r^{\prime}}$, the above integral finite and equals to $\frac{1}{d_{k}-\beta r^{\prime}} (\frac{|x|}{2})^{d_{k}-\beta r^{\prime}}$. Therefore, $I_{1}(x)\leq \frac{(2t)^{-\frac{d_{k}}{2}r^{\prime}}}{(d_{k}-\beta r^{\prime})2^{d_{k}-\beta r^{\prime}}}\frac {|x|^{d_{k}}}{e^{\frac{r^{\prime}}{16t}|x|^{2}}}$. Since $\frac {|x|^{d_{k}}}{e^{\frac{r^{\prime}}{16t}|x|^{2}}}\longrightarrow 0$ as $|x|\longrightarrow\infty$, it follows that $I_{1}(x)\leq A^{\prime}_{d,r,\beta,t,k}$ for some constant $A^{\prime}_{d,r,\beta,t,k}>0$.

Now we consider the integral $I_{2}$. As $|y|>\frac{|x|}{2}, \big(\frac{|x|}{|y|}\big)^{\beta r^{\prime}}< 2^{\beta r^{\prime}}$ for $\beta,r^{\prime}>0$. 
\begin{eqnarray*}
I_{2}(x)\leq 2^{\beta r^{\prime}}\int\limits_{|y|>\frac{|x|}{2}} |\tau_{y}^{k}q_{t}^{k}(x)|^{r^{\prime}}w_{k}(y)dy
&\leq& 2^{\beta r^{\prime}} \int\limits_{\mathbb{R}^{d}} |q_{t}^{k}(y)|^{r^{\prime}}w_{k}(y)dy\\
&\leq& 2^{\beta r^{\prime}} B_{d,r,k} t^{-\frac{d_{k}}{2}(r^{\prime}-1)}=B^{\prime}_{d,r,\beta,t,k},
\end{eqnarray*}
using (\ref{5peq32}). Hence for $|x|>1$, from (\ref{5peq33}) , we get a constant $C^{\prime}_{d_{k},r^{\prime},t}>0$ such that 
\begin{eqnarray}\label{5peq35}
||x|^{w}e^{t\Delta_{k}}f(x)|\leq C^{\prime}_{d,r,\beta,t,k} \||y|^{\beta}u\|_{r,w_{k}}.
\end{eqnarray}
Consequently, combining (\ref{5peq34}) and (\ref{5peq35}), we get a constant $D_{d,r,\beta,t,k}>0$ such that 
\begin{eqnarray*}
 \||x|^{w}e^{t\Delta_{k}}f\|_{L^{\infty}(\mathbb{R}^{d})}\leq D_{d,r,\beta,t,k} \||x|^{\beta}f\|_{r,w_{k}},
 \end{eqnarray*}
 thus proving (ii). 
 
 Next we shall prove (iii). Consider 
 \begin{eqnarray*}
 \left|\frac{\partial}{\partial x_{i}}(e^{t\Delta_{k}}f)(x)\right|
 &=&\left|\frac{\partial}{\partial x_{i}}(f*_{k}q_{t}^{k})(x)\right|\\
 &=& \left|\frac{\partial}{\partial x_{i}}\int\limits_{\mathbb{R}^{d}}f(y)\tau^{k}_{y}q_{t}^{k}(x)w_{k}(y)dy\right|\\
 &=& \left|\int\limits_{\mathbb{R}^{d}}f(y)\frac{\partial}{\partial x_{i}}\tau^{k}_{y}q_{t}^{k}(x)w_{k}(y)dy\right|\\
 &\leq& \int\limits_{\mathbb{R}^{d}} |f(y)|\left|\frac{\partial}{\partial x_{i}}\tau^{k}_{y}q_{t}^{k}(x)\right|w_{k}(y)dy\\
 &\leq& C t^{-\frac{1}{2}}\int\limits_{\mathbb{R}^{d}} |f(y)||\tau^{k}_{y}q_{2t}^{k}(x)| w_{k}(y) dy,
 \end{eqnarray*}
 by taking $m=0,\alpha=e_{i},\beta=0$ in (\ref{5peq54}). Now using (i), there exists constant $E_{d,r,\beta,t,k}>0$ such that
 \begin{eqnarray*}
 \|\partial_{x_{i}}e^{t\Delta_{k}}f\|_{L^{\infty}(\mathbb{R}^{d})}\leq E_{d,r,\beta,t,k}\||x|^{\beta}f\|_{r,w_{k}},
 \end{eqnarray*}
 thus proving (iii).     
\end{pf}

Using Proposition \ref{5ppr2}, we can prove the following theorem. 
\begin{thm}\label{5pth8}
Let $d\geq 2, 1< r<\infty$. If $0<\beta<\frac{d_{k}}{r^{\prime}}$, then for any fixed $t>0$, the operator $e^{t\Delta_{k}}$ is a compact operator from $L^{r}(\mathbb{R}^{d},|x|^{\beta r}w_{k})$ to $L^{\infty}(\mathbb{R}^{d})$. 
\end{thm}
\begin{proof}
Let $(u_j)_{j\in \mathbb{N}} \in L^r(\mathbb{R}^d,|x|^{\beta r}w_k)$ be a bounded sequence so that $\||x|^\beta u_j\|_{r,w_k}\leq C_0$ for all $j\in \mathbb{N}$. We will prove that the sequence $(e^{t\Delta_k}u_j)_{j\in \mathbb{N}}$ has convergent subsequence in $L^{\infty}(\mathbb{R}^{d})$ and that will imply the theorem. Let $v_j=e^{t\Delta_k}u_j$. Since $(u_j)_{j\in \mathbb{N}}$ is a bounded sequence, using Proposition \ref{5ppr1} (i), we have
\begin{align*}\label{eqn 1}
  \|v_j\|_{L^\infty(\mathbb{R}^d)}=  \|e^{t\Delta_k}u_j\|_{L^\infty(\mathbb{R}^d)}\leq C\||x|^\beta u_j\|_{r,w_k}\leq C_0.
\end{align*}
This proves that each $v_{j}\in L^{\infty}(\mathbb{R}^{d})$ and the collection $(v_j)_j$ is equibounded in $\mathbb{R}^{d}$. 
Moreover, Proposition \ref{5ppr1} (iii) shows that $(v_j)_j$ is also equicontinuos in $\mathbb{R}^{d}$. Now for each $n \in \mathbb{N}$, we define the compact set $A_n:=\{x\in \mathbb{R}^d:|x|\leq n\}$. Then by the Arzel\'a-Ascoli theorem for each $n\in\mathbb{N}$, there exists a subsequence of $(v_j)_j$ which converges uniformly in $A_n$. Now by applying diagonal argument, we get a subsequence of $(v_j)_j$ which converges uniformly in every $A_n$. let us call this subsequence also by $(v_j)_j$ and we can write $v_j \to v$ uniformly in each $A_n$ for some $v \in L^{\infty}(\mathbb{R}^{d})$. 

Now let $z$ be such that $0<z <\beta$. By Proposition \ref{5ppr1}(ii), we can write
\begin{eqnarray}\label{eqn 2}
    \||x|^z v_j\|_{L^\infty(\mathbb{R}^d)}=\|| x|^z e^{t\Delta_k}u_j\|_{L^\infty(\mathbb{R}^d)}\leq C_1\||x|^\beta u_j\|_{r,w_k}.
\end{eqnarray}
Now by using (\ref{eqn 2}) we get
\begin{align*}
    \sup_{|x|>n}|v_j|\leq \sup_{|x|>n}\Big(\frac{|x|}{n}\Big)^z|v_j|\leq n^{-z}\||x|^z v_j\|_{L^\infty(\mathbb{R}^d)}\leq C_1n^{-z}.
\end{align*}
Thus we can easily see that $v_j \to v$ strongly in $L^\infty(\mathbb{R}^d)$ and this proves the theorem.
\end{proof}

\section{Improved Stein-Weiss inequality for the D-Riesz potential} \label{5pth5}
In this section, we will be focusing on deriving an improved version of the Stein-Weiss inequality for the D-Riesz potential. In other words, we are interested to generalize Theorem\ref{5pth1}. Towards this, for any $\delta >0$, first we define the Dunkl Besov space as
\begin{equation*}
 \dot{B}_{\infty,\infty}^{-\delta,k}:=\{f:f ~ \text{is a tempered distribution on $\mathbb{R}^{d}$ and}~ \|f\|_{\dot{B}_{\infty,\infty}^{-\delta,k}}<\infty \},
\end{equation*}
where 
\begin{equation}\label{Besov norm}
    \|f\|_{\dot{B}_{\infty,\infty}^{-\delta,k}}:= \sup_{t>0}t^{\delta/2}\|e^{t\Delta_k}f\|_{L^{\infty}}.
\end{equation}
\begin{thm}\label{5pth6}
Let $d\geq2$ and $0<\alpha<d_k$. Also let $\beta$, $\gamma$, $\mu$, $\theta$, $r$ and $s$ be such that  $1<r< s<\infty$, $\gamma >-\frac{d_k}{s}$, $\beta <\frac{d_k}{r'}$, $\beta \geq\frac{\gamma}{\theta}$. Also it satisfy that $\mu>0$, $\max \{\frac{r}{s}, \frac{\mu}{\mu+\alpha}\}<\theta \leq 1 $ and
\begin{eqnarray}\label{5peq43}
\frac{d_k}{s}+\gamma=\Big(\beta +\frac{d_k}{r}-\alpha\Big)\theta+\mu(1-\theta). 
\end{eqnarray}
Then for all $f\in L^r(\mathbb{R}^d,|x|^{\beta r}w_k)\cap \dot{B}_{\infty, \infty}^{-\mu -\alpha,k}$,the following inequality holds:
\begin{equation}\label{5peq44}
    \||x|^\gamma I_{\alpha}^{k} f\|_{L^s(\mathbb{R}^d,w_k)}\leq C_{k}\|x|^\beta f \|_{L^r(\mathbb{R}^d,w_k)}^\theta \|f\|_{\dot{B}_{\infty,\infty}^{-\mu -\alpha,k}}^{1-\theta}.
\end{equation}
\end{thm}
\begin{proof}
The case $\theta =1$ reduces to Theorem \ref{5pth1}. So we will prove the theorem for $\theta <1$.
For $f \in L^r(\mathbb{R}^d,|x|^{\beta r}w_k)$, let $u=I_{\alpha}^{k} f$. Then $u$ has an integral representation of the following form:
\begin{equation*}
    u=\frac{1}{\Gamma(\alpha/2)}\int_0^\infty t^{\alpha/2-1}e^{t\Delta_k}fdt.
\end{equation*}
Now, we can write $u=H_{k}f+L_{k}f$, where
\begin{equation*}
    H_{k}f:=\frac{1}{\Gamma(\alpha/2)}\int_0^Tt^{\alpha/2-1}e^{t\Delta_k}fdt ~~ \text{and} ~~ L_{k}f :=\frac{1}{\Gamma(\alpha/2)}\int_T^\infty t^{\alpha/2-1}e^{t\Delta_k}fdt,
\end{equation*}
for some $T>0$ to be chosen later. 

Our aim is to find a bound for $u$. To achieve this we will look for the bounds for $L_{k}f$ and $H_{k}f$ separately.
Using the definition of Besov norm in (\ref{Besov norm}), we have 
\begin{eqnarray}\label{5peq75}
|L_{k}f(x)|\leq C_{k}T^{-\mu/2}\|f\|_{\dot{B}_{\infty,\infty}^{-\mu-\alpha,k}}.
\end{eqnarray}
Now we will find the bound for $H_{k}f$.

Let 
\begin{equation*}
    \Phi_{\alpha,T}^k(x)=\frac{1}{\Gamma(\frac{\alpha}{2})}\int_0^Tt^{\frac{\alpha}{2}-1}q_t^k(x)dt.
\end{equation*}
It can be easily verified that $H_{k}f=\Phi_{\alpha,T}^k*_kf$. Fix $\epsilon=\frac{\mu/\theta-\mu}{2}>0$. We note that since $\theta >\frac{\mu}{\mu+\alpha},~\alpha-2\epsilon>0$.

Since for a given $u>0$, there exists a constant $C>0$ such that for any non-zero real $x$, $e^{-|x|}\leq \frac{C}{|x|^{u}}$ holds, we can write by taking
$u=(d_k-\alpha)/2+\epsilon>0$, 
\begin{align*}
     \Phi_{\alpha,T}^k(x)
     = \frac{1}{\Gamma(\frac{\alpha}{2})}\int_0^Tt^{\frac{\alpha}{2}-1}(2t)^{-\frac{d_{k}}{2}} e^{-\frac{|x|^{2}}{4t}}dt
     &\leq C\int_{0}^Tt^{(\alpha-d_k)/2-1}\bigg(\frac{4t}{|x|^2}\bigg)^{(d_k-\alpha)/2+\epsilon}dt\\
     &\leq C_{k}\frac{1}{|x|^{d_k-\alpha+2\epsilon}}\int_0^Tt^{-1+\epsilon}dt\\
     &= C_{k} \frac{T^\epsilon}{|x|^{d_k-\alpha+2\epsilon}}.
\end{align*}
Since Dunkl translation is linear and  positivity-preserving for radial functions we can write $ \tau_y^k( \Phi_{\alpha,T}^k)(x) \leq T^\epsilon \tau_y^k(|.|^{-(d_k-\alpha+2\epsilon)})(x)$. So we obtain $H_{k}f(x)\leq C_{k}T^\epsilon I_{\alpha-2\epsilon}^kf(x)$. Choose $T$ such that $T^{-\mu/2}\|f\|_{\dot{B}_{\infty,\infty}^{-\mu-\alpha,k}} =T^\epsilon I_{\alpha-2\epsilon}^kf(x)$, which implies that
\begin{equation*}
    T=\bigg(\frac{\|f\|_{\dot{B}_{\infty,\infty}^{-\mu-\alpha,k}}}{I^k_{\alpha-2\epsilon }f(x)}\bigg)^{1/(\epsilon+\mu/2)}.
\end{equation*}
By substituting the value of T in (\ref{5peq75}), we arrive at the pointwise bound
\begin{align*}
     |u(x)|=| I_{\alpha}^{k} f(x)| &\leq C_{k} I^{k}_{\alpha-2\epsilon}f(x)^\theta \|f\|^{1-\theta}_{\dot{B}_{\infty,\infty}^{-\mu-\alpha,k}}.
\end{align*}
Hence
\begin{equation}\label{Stein-Weiss Gorbachev 1}
    \||x|^\gamma I_{\alpha}^{k} f\|_{L^s(\mathbb{R}^d,w_k)}  \leq C_{k}\||x|^{\gamma/\theta}I^{k}_{\alpha-2\epsilon}f\|^\theta_{L^{s\theta}(\mathbb{R}^d,w_k)}\|f\|^{1-\theta}_{\dot{B}_{\infty,\infty}^{-\mu-\alpha,k}}.
\end{equation}
If we assume that $\alpha^{\prime}=\alpha-2\epsilon, \gamma^{\prime}=\frac{\gamma}{\theta},s^{\prime}=s\theta$, then it is easy to see by the hypothesis of the theorem and the choice of $\epsilon$ that $\gamma^{\prime}>-\frac{d_{k}}{s^{\prime}},\beta\geq\gamma^{\prime},0<\alpha^{\prime}<d_{k}, r< s^{\prime}, \beta < \frac{d_k}{r'}$ and $\alpha^{\prime}+\gamma^{\prime}-\beta=d_{k}(\frac{1}{r}-\frac{1}{s^{\prime}})$.
Hence by using Theorem \ref{5pth1}, we get 
\begin{equation}\label{Stein-Weiss Gorbachev 2}
    \||x|^{\gamma/\theta}I^{k}_{\alpha-2\epsilon}f\|_{L^{s\theta}(\mathbb{R}^d,w_k)} \leq C_{k}^{\prime}\||x|^\beta f\|_{L^r(\mathbb{R}^d,w_k)}.
\end{equation}
Now from (\ref{Stein-Weiss Gorbachev 1}) and (\ref{Stein-Weiss Gorbachev 2}), we get the desired inequality

\begin{equation*}
     \||x|^\gamma I_{\alpha}^{k} f\|_{L^s(\mathbb{R}^d,w_k)}\leq D_{k}\||x|^\beta f\|^\theta_{L^r(\mathbb{R}^d, w_k)}\|f\|^{1-\theta}_{\dot{B}_{\infty,\infty}^{-\mu-\alpha,k}}.
\end{equation*}
\end{proof}

\begin{rem}
One can prove Theorem \ref{5pth6} for the case $\theta=\frac{\mu}{\mu+\alpha}$ if the weighted $L^{r}$-boundedness of the maximal function $M_{k}$ is known for $1<r<\infty$. We recall that for $f\in\mathcal {S} (\mathbb{R}^{d})$, S. Thangavelu and Y. Xu \cite{sy},  defined the maximal function $M_{k}$ as follows: 
\begin{eqnarray*}
M_{k}f(x)=\sup_{r>0} \frac{\bigg|\int\limits_{\mathbb{R}^{d}} f(y)\tau_{x}^{k}\chi_{B_{r}}(y)w_{k}(y)dy\bigg|}{\int\limits_{B_{r}}w_{k}(y)dy},
\end{eqnarray*}
where $B_{r}=\{y\in\mathbb{R}^{d}:|y|\leq r\}$. When $\theta=\frac{\mu}{\mu+\alpha}$, following the proof of \cite {pd} and the fact that $|H_{k}f(x)|\leq C_{k}T^{\frac{\alpha}{2}}M_{k}f(x)$, one can show that
\begin{eqnarray*}
\||x|^\gamma I_{\alpha}^{k} f\|_{L^{s}(\mathbb{R}^{d},w_{k})}
&\leq C_{k}\||x|^{\beta} M_{k}f\|^\theta_{L^{r}(\mathbb{R}^{d},w_{k})}\|f\|^{1-\theta}_{\dot{B}_{\infty,\infty}^{-\mu-\alpha,k}}.
\end{eqnarray*}
Now in order to obtain (\ref{5peq44}), one need to prove
\begin{eqnarray}\label{5peq69}
\||x|^{\beta} M_{k}f\|_{L^{r}(\mathbb{R}^{d},w_{k})}\leq \||x|^{\beta} f\|_{L^{r}(\mathbb{R}^{d},w_{k})},
\end{eqnarray}
which is not known to be true in general. For $\beta=0$, (\ref{5peq69}) has been proved by S. Thangavelu and Y. Xu in \cite{sy}.
\end{rem}

\section {Existence of an extremal of Stein-Weiss inequality for the D-Riesz potential}
The aim of this section is to prove the existence of a maximizer for the inequality (\ref{5peq14}) stated in Theorem \ref{5pth1}. When $k\equiv 0$ (that is for Theorem \ref{5pnth2}), the existence of a maximizer is proved by P. D. Napoli et al. in \cite {pd}. Towards this, we first prove the following embedding theorem. 
\begin{thm}\label{5pth2}
Let $d\in \mathbb{N}, 1< r\leq s<\infty,\gamma >-\frac{d_{k}}{s},\beta\geq \gamma,0<\alpha<d_{k}, \beta<\frac{d_{k}}{r^{\prime}}$. Further we assume that 
\begin{eqnarray}\label{5peq12}
\beta+\frac{d_{k}}{r}>\alpha>\frac{d_{k}}{r}-\frac{d_{k}}{s}+\beta-\gamma>0.
\end{eqnarray}
Then if $\mathcal{K}\subset\mathbb{R}^{d}$ is compact, then one has the compact embedding
\begin{eqnarray}\label{5peq16}
\dot{H}_{\beta,k}^{\alpha,r}(\mathbb{R}^{d})\subset L^{s}(\mathcal{K},|x|^{\gamma s}w_{k}),
\end{eqnarray}
where the space $\dot{H}_{\beta,k}^{\alpha,r}(\mathbb{R}^{d})$ is defined in (\ref{5peq40}). 
\end{thm}

\begin{pf}
Let $u\in \dot{H}_{\beta,k}^{\alpha,r}(\mathbb{R}^{d})$. Then $u=I^{k}_{\alpha}f$, for some $f\in L^{r}(\mathbb{R}^{d},|x|^{\beta r}w_{k})$. 
Now we choose $\tilde{s}$ such that
\begin{eqnarray}\label{5peq13}
\frac{1}{\tilde{s}} (d_{k}+\gamma s)=\frac{d_{k}}{r}+\beta-\alpha. 
\end{eqnarray}
We define $v=\frac{\tilde{s}}{s}$ and $\tilde{\gamma}=\frac{\gamma s}{\tilde{s}}$. From (\ref{5peq12}), it follows that $v>1$ and $\tilde{\gamma}=\frac{\gamma}{v}$. Then (\ref{5peq13}) can be rewritten as 
\begin{eqnarray*}
d_{k}\bigg(\frac{1}{r}-\frac{1}{\tilde{s}}\bigg)=\alpha+\tilde{\gamma}-\beta. 
\end{eqnarray*}
We replace $\gamma$ and $s$ in Theorem \ref{5pth1} by $\tilde{\gamma}$ and $\tilde{s}$ respectively. Since $v>1$, $r\leq s$ implies that $r\leq sv=\tilde{s}$ and $\beta\geq\gamma$ implies $\beta\geq\tilde{\gamma}$. Also $\tilde{\gamma}>-\frac{d_{k}}{\tilde{s}}$ since $\gamma>-\frac{d_{k}}{s}$. Thus all the conditions of Theorem \ref{5pth1} are satisfied and hence from (\ref{5peq14}), we have 
\begin{eqnarray}\label{5peq15}
\||x|^{\tilde{\gamma}}I_{\alpha}^{k}f\|_{L^{\tilde{s}}(\mathbb{R}^{d},w_{k})}\leq C_{k} \||x|^{\beta}f\|_{L^{r}(\mathbb{R}^{d},w_{k})}.
\end{eqnarray}
Applying Holder's inequality with components $v$ and $v^{\prime}$,
\begin{eqnarray*}
\int\limits_{\mathcal{K}} |u(x)|^{s}|x|^{\gamma s} w_{k}(x)dx
&=&\int\limits_{\mathcal{K}} |u(x)|^{s} |x|^{\frac{\gamma s}{v}} |x|^{\frac{\gamma s}{v^{\prime}}} w_{k}(x)dx\\
&\leq&  \bigg(\int\limits_{\mathcal{K}} |u(x)|^{sv} |x|^{\gamma s} w_{k}(x)dx\bigg)^{\frac{1}{v}}\bigg(\int\limits_{\mathcal{K}}|x|^{\gamma s}w_{k}(x)dx\bigg)^{\frac{1}{v^{\prime}}}\\
&\leq& C_{\mathcal{K}} \bigg(\int\limits_{\mathcal{K}} |u(x)|^{\tilde{s}} |x|^{\tilde{\gamma} \tilde{s}} w_{k}(x)dx\bigg)^{\frac{1}{v}},
\end{eqnarray*}
since by property (2) of Section 2, the second integral is finite under the assumption  $\gamma>-\frac{d_{k}}{s}$.
Then using (\ref{5peq15}),
\begin{eqnarray*}
\bigg(\int\limits_{\mathcal{K}} |u(x)|^{s}|x|^{\gamma s} w_{k}(x)dx\bigg)^{\frac{1}{s}}
&\leq& C_{\mathcal{K}} \bigg(\int\limits_{\mathcal{K}} |I_{\alpha}^{k}f(x)|^{\tilde{s}} |x|^{\tilde{\gamma} \tilde{s}} w_{k}(x)dx\bigg)^{\frac{1}{\tilde{s}}}\\
&\leq& C_{\mathcal{K}}\||x|^{\beta}f\|_{L^{r}(\mathbb{R}^{d},w_{k})}=C_{\mathcal{K}} \|u\|_{\dot{H}_{\beta,k}^{\alpha,r}(\mathbb{R}^{d})},
\end{eqnarray*}
proving that the embedding (\ref{5peq16}) is continuous. 

Let us define the kernel of the D-Riesz potential as
\begin{eqnarray*}
K_{\alpha,k}(x)= (c_{\alpha}^{k})^{-1} |x|^{-(d_{k}-\alpha)}
\end{eqnarray*}
and for $t>0$, the truncated kernel as
\begin{eqnarray*}
K_{\alpha,k}^{t}(x)=(c_{\alpha}^{k})^{-1} |x|^{-(d_{k}-\alpha)}\chi_{\{|x|>t\}}.
\end{eqnarray*}
Now following similarly as in the proof of \cite[Lemma 4.2]{pd} and using (\ref{5peq14}), we can prove the following Lemma.
\begin{lem}\label {5plem1}
With the same conditions as that of Theorem \ref{5pth2}, let
\begin{eqnarray*}
\delta=\alpha-\bigg(\frac{d_{k}}{r}-\frac{d_{k}}{s}+\beta-\gamma\bigg).
\end{eqnarray*}
Then for any $f\in L^{r}(\mathbb{R}^{d},|x|^{\beta r}w_{k})$ and for any $t>0$,
\begin{eqnarray*}
\|\big(K_{\alpha,k}^{t}*_{k}f-K_{\alpha,k}*_{k}f\big)|x|^{\gamma}\|_{s,w_{k}}\leq C t^{\delta} \||x|^{\beta}f\|_{r,w_{k}}. 
\end{eqnarray*} 
\end{lem}
Now we shall show that the embedding (\ref{5peq16}) is compact. 

Let $\{u_{m}\}$ be a bounded sequence in $\dot{H}_{\beta,k}^{\alpha,r}(\mathbb{R}^{d})$. Then we can write $u_{m}=I_{\alpha}^{k}f_{m}$, where $\{f_{m}\}$ is also a bounded sequence in $L^{r}(\mathbb{R}^{d},|x|^{\beta r}w_{k})$. Since $L^{r}(\mathbb{R}^{d},|x|^{\beta r}w_{k})$ is a reflexive space, $\{f_{m}\}$ has a subsequence, denoted by $f_{m}$ itself such that $f_{m}$ converges weakly to a function $f$ in $L^{r}(\mathbb{R}^{d},|x|^{\beta r}w_{k})$. Let $u=I_{\alpha}^{k}f$. It is easy to see that $u=K_{\alpha,k}*_{k}f$. Now let us assume that $u_{m}^{t}=K_{\alpha,k}^{t}*_{k}f_{m}$ and $u^{t}=K_{\alpha,k}^{t}*_{k}f$. Consider 
\begin{eqnarray*}
\|(u_{m}-u)|x|^{\gamma}\|_{L^{s}(\mathcal{K},w_{k})}\leq \|(u_{m}-u_{m}^{t})|x|^{\gamma}\|_{L^{s}}+\|(u_{m}^{t}-u^{t})|x|^{\gamma}\|_{L^{s}}+\|(u^{t}-u)|x|^{\gamma}\|_{L^{s}}.
\end{eqnarray*}
Using Lemma \ref{5plem1}, 
\begin{eqnarray*}
\|(u_{m}-u_{m}^{t})|x|^{\gamma}\|_{L^{s}(\mathcal{K},w_{k})}= \|(K_{\alpha,k}*_{k}f_{m}-K_{\alpha,k}^{t}*_{k}f_{m})|x|^{\gamma}\|\leq Ct^{\delta} \||x|^{\beta}f_{m}\|_{r,w_{k}}\leq Dt^{\delta}.
\end{eqnarray*}
Similarly, 
\begin{eqnarray*}
\|(u^{t}-u)|x|^{\gamma}\|_{L^{s}(\mathcal{K},w_{k})}\leq Ct^{\delta}\|x^{\beta}f\|_{r,w_{k}}\leq D t^{\delta}. 
\end{eqnarray*}
Choose $\epsilon>0$. For very small $t>0$, each of the above estimates can be made less that $\frac{\epsilon}{3}$ for all m. We are left to get bound for $\|(u_{m}^{t}-u^{t})|x|^{\gamma}\|_{L^{s}(\mathcal{K},w_{k})}$. 

We note that the function $K_{\alpha,k}^{t}$ is a radial function with no singular point and therefore, by applying (\ref{eqn2}) for the function $K_{\alpha,k}^{t}$, we get
\begin{eqnarray}
\tau_{y}^{k}K_{\alpha,k}^{t}(x)
&=&\int\limits_{\mathbb{R}^{d}}K_{\alpha,k}^{t}(\sqrt{|x|^{2}+|y|^{2}-2\langle y,\eta\rangle})d\mu_{x}^{k}(\eta)\nonumber\\
&=& \int\limits_{\mathbb{R}^{d}} \frac{(c_{\alpha}^{k})^{-1}}{(|x|^{2}+|y|^{2}-2\langle y,\eta\rangle)^{\frac{d_{k}-\alpha}{2}}}\chi_{\{\eta:|x|^{2}+|y|^{2}-2\langle y,\eta\rangle\geq t^{2}\}}{(\eta)}d\mu_{x}^{k}(\eta)\nonumber\\
&=& \int\limits_{\eta:A(x,y,\eta)\geq t}\frac {(c_{\alpha}^{k})^{-1}}{(A(x,y,\eta))^{d_{k}-\alpha}} d\mu_{x}^{k}(\eta),\label{5peq41}
\end{eqnarray}
where $A(x,y,\eta)=\sqrt{|x|^{2}+|y|^{2}-2\langle y,\eta\rangle}$. Then
\begin{eqnarray}
\tau_{y}^{k}K_{\alpha,k}^{t}(x)
&\leq& \frac{(c_{\alpha}^{k})^{-1}}{t^{d_{k}-\alpha}}d\mu_{x}^{k}\{{\eta: A(x,y,\eta)\geq t}\}
\leq \frac{(c_{\alpha}^{k})^{-1}}{t^{d_{k}-\alpha}},\label{5peq18}
\end{eqnarray}
since $d\mu_{x}^{k}$ is a probability measure. 

It is proved in \cite {ba} that 
\begin{eqnarray}\label{5peq26}
\min\limits_{g\in G}|g.x-y|\leq A(x,y,\eta)\leq \max\limits_{g\in G} |g.x-y|, \forall~ x,y\in\mathbb{R}^{d} ~and~\eta\in co (G.x).
\end{eqnarray}
Let $y\in \mathcal{K}$, where $\mathcal{K}$ is a compact set in $\mathbb{R}^{d}$. Then there exists $R>0$ such that $\mathcal{K}\subset B(0,R)$. 

Consider 
\begin{eqnarray}
&&\int\limits_{\mathbb{R}^{d}} |\tau_{y}^{k}K_{\alpha,k}^{t}(x)|^{r^{\prime}}|x|^{-\beta r^{\prime}}w_{k}(x)dx\nonumber\\
&=& \int\limits_{|x|\leq 2R}| \tau_{y}^{k}K_{\alpha,k}^{t}(x)|^{r^{\prime}}|x|^{-\beta r^{\prime}}w_{k}(x)dx +\int\limits_{|x|> 2R} |\tau_{y}^{k}K_{\alpha,k}^{t}(x)|^{r^{\prime}}|x|^{-\beta r^{\prime}}w_{k}(x)dx\nonumber\\
&=& I_{1}(y)+I_{2}(y). \label{5peq19}
\end{eqnarray}
Substituting the bound for $\tau_{y}^{k}K_{\alpha,k}^{t}$ from (\ref{5peq18}) in $I_{1}(y)$, we have
\begin{eqnarray*}
I_{1}(y)\leq \frac{(c_{\alpha}^{k})^{-r^{\prime}}}{t^{r^{\prime}(d_{k}-\alpha)}}\int\limits_{|x|\leq 2R}|x|^{-\beta r^{\prime}}w_{k}(x)dx=\frac{(c_{\alpha}^{k})^{-r^{\prime}}}{t^{r^{\prime}(d_{k}-\alpha)}}\frac{a_{k}(2R)^{d_{k}-\beta r^{\prime}}}{d_{k}-\beta r^{\prime}}= M_{1},
\end{eqnarray*}
using property (2) of section 2 provided $\beta r^{\prime}<d_{k}$ , that is, if $\beta<\frac{d_{k}}{r^{\prime}}$. On the other hand, for the integral $I_{2}(y)$, substituting the value of  $\tau_{y}^{k}K_{\alpha,k}^{t}$ from (\ref{5peq41}) and then using (\ref{5peq26}), we get
\begin{eqnarray}
I_{2} (y) &=&\int\limits_{|x|>2R} \left|\int\limits_{\eta:A(x,y,\eta)\geq t}\frac{(c_{\alpha}^{k})^{-1}}{(A(x,y,\eta))^{d_{k}-\alpha}}d\mu_{x}^{k}(\eta)\right|^{r^{\prime}}|x|^{-\beta r^{\prime}} w_{k}(x)dx\nonumber\\
&\leq& \int\limits_{|x|>2R} \left|\frac{(c_{\alpha}^{k})^{-1}}{\bigg(\min\limits_{g\in G}|gx-y|\bigg)^{d_{k}-\alpha}}\right|^{r^{\prime}} |x|^{-\beta r^{\prime}}w_{k}(x)dx\nonumber\\
&=& (c_{\alpha}^{k})^{-r^{\prime}}\int\limits_{|x|>2R} \frac{|x|^{-\beta r^{\prime}}}{\bigg(\min\limits_{g\in G}|g.x-y|\bigg)^{r^{\prime}(d_{k}-\alpha)}}w_{k}(x)dx.\label{5peq20}
\end{eqnarray}
Since g is a reflection, we observe that 
\begin{eqnarray*}
|g.x-y|\geq |g.x|-|y|=|x|-|y|,~\forall~ g~\in G.
\end{eqnarray*}
If $y\in\mathcal{K}$ and $x$ is in the integration region of $I_{2}(y)$, then $|y|\leq R\leq \frac {|x|}{2}$. So $|g.x-y|\geq \frac{|x|}{2}$. Therefore, $\min\limits_{g\in G} |g.x-y|\geq \frac {|x|}{2}$ for all $y\in\mathcal{K}$. From (\ref{5peq20}),
\begin{eqnarray*}
I_{2}(y)\leq (c_{\alpha}^{k})^{-r^{\prime}} 2^{r^{\prime}(d_{k}-\alpha)} \int\limits_{|x|>2R} \frac {|x|^{-\beta r^{\prime}}}{|x|^{r^{\prime}(d_{k}-\alpha)}}w_{k}(x)dx\leq M_{2},
\end{eqnarray*}
if $d_{k}<\beta r^{\prime}+r^{\prime}(d_{k}-\alpha)$ i.e., if $\alpha<\beta+\frac{d_{k}}{r}$. Thus from (\ref{5peq19}),
\begin{eqnarray}\label{5peq21}
\int\limits_{\mathbb{R}^{d}} |\tau_{y}^{k}K_{\alpha,k}^{t}(x)|^{r^{\prime}}|x|^{-\beta r^{\prime}}w_{k}(x)dx\leq M_{1}+M_{2},~\forall~ y\in\mathbb{R}^{d}. 
\end{eqnarray}
In particular, we have proved that $\tau_{y}^{k}K_{\alpha,k}^{t}\in L^{r^{\prime}}(\mathbb{R}^{d},|x|^{-\beta r^{\prime}} w_{k}),~\forall~y\in\mathcal{K}$. Since $f_{m}$ converges weakly to $f$ in $L^{r}(\mathbb{R}^{d},|x|^{\beta r} w_{k})$, $K_{\alpha,k}^{t}*_{k}f_{m}(y)$ converges to $K_{\alpha,k}^{t}*_{k}f(y)$ as a sequence of complex numbers. Thus $u_{m}^{t}(y)$ converges to $u^{t}(y)$ for all $y\in\mathcal{K}$. Moreover, since
\begin{eqnarray*}
u_{m}^{t}(y)=K_{\alpha,k}^{t}*_{k} f_{m} (y)=\int\limits_{\mathbb{R}^{d}} f_{m}(x)\tau_{x}^{k}K_{\alpha,k}^{t}(y)w_{k}(x)dx,
\end{eqnarray*}
using (\ref{5peq21}) and the fact that $f_{m}$ is a bounded sequence in $L^{r}(\mathbb{R}^{d}, |x|^{\beta r}w_{k})$, 
\begin{eqnarray*}
|u_{m}^{t}(y)|\leq \bigg(\int\limits_{\mathbb{R}^{d}}|f_{m}(x)|^{r}|x|^{\beta r}w_{k}(x)dx\bigg)^{\frac{1}{r}} \bigg(\int\limits_{\mathbb{R}^{d}}|\tau_{x}^{k}K_{\alpha,k}^{t}(y)|^{r^{\prime}}|x|^{-\beta r^{\prime}}w_{k}(x)dx\bigg)^{\frac{1}{r^{\prime}}} \leq A. 
\end{eqnarray*}
By Lebesgue Dominated convergence theorem, $\|u_{m}^{t}-u^{t}|x|^{\gamma}\|_{L^{s}(\mathcal{K},w_{k})}$ converges to zero (as the weight $|x|^{\gamma s}w_{k}$ is integrable on $\mathcal{K}$ under the condition $\gamma>-\frac{d_{k}}{s}$). Hence we can make it less than $\frac{\epsilon}{3}$ for large m. Thus $\|(u_{m}-u)|x|^{\gamma}\|_{L^{s}(\mathcal{K},w_{k})}\leq\epsilon$ for large m, proving that $u_{m}$ converges to $u$ strongly on $L^{s}(\mathcal{K},|x|^{\gamma s}w_{k})$. This completes the proof that the embedding (\ref{5peq16}) is compact. 
\end{pf}

Now using Theorem \ref{5pth6} and Theorem \ref{5pth2}, we are ready to prove our main result that (\ref{5peq27}) has a maximizer. 
\begin{thm}\label{5pth7}
Let $d\geq 2, 2<s<\infty,-\frac{d_{k}}{s}<\gamma<\beta,0<\beta<\frac{d_{k}}{2}$ and the relation
\begin{eqnarray}\label{5peq45}
\frac{1}{s}-\frac{1}{2}=\frac{\beta-\gamma-\alpha}{d_{k}}
\end{eqnarray}
holds. Then there exists a maximizer for $W_{k}$.  
\end{thm}
\begin{pf}
Let $\{f_j\}_{j\in \mathbb{d}}$ be a maximizing sequence for $W_k$, which we can take to be normalized i.e., 
\begin{eqnarray}\label{5peq48}
\||x|^{\beta}f_{j}\|_{L^{2}(\mathbb{R}^{d},w_{k})}=1~and~ \||x|^{\gamma}I_{\alpha}^{k}f_{j}\|_{L^{s}(\mathbb{R}^{d},w_{k})}\to W_{k}. 
\end{eqnarray}
Now, in Theorem \ref{5pth6}, we set 
\begin{eqnarray}\label{5peq53}
\mu=\frac{d_{k}}{2}+\beta-\alpha,
\end{eqnarray}
and choose $\theta$ such that 
\begin{eqnarray*}
\max\bigg\{\frac{2}{s},\frac{\mu}{\mu+\alpha},\frac{\gamma}{\beta}\bigg\}<\theta<1.
\end{eqnarray*}
Because of relation (\ref{5peq45}), equation (\ref{5peq43}) holds for this particular choice of $\mu$. It is also easy to see that remaining conditions of Theorem \ref{5pth6} hold under the hypothesis of the Theorem and the choice of $\mu,\theta$. Also since relation (\ref{5peq53}) holds, from (\ref{5peq46}),
\begin{eqnarray*}
 \|f_{j}\|_{\dot{B}_{\infty,\infty}^{-\mu-\alpha,k}}
 =\sup\limits_{t>0} t^{\frac{\mu+\alpha}{2}} \|e^{t\Delta_{k}}f_{j}\|_{L^{\infty}}
&=&\sup\limits_{t>0} t^{\frac{1}{2}(\frac{d_{k}}{2}+\beta)}\|e^{t\Delta_{k}}f_{j}\|_{L^{\infty}}\\
&\leq& C_{d,\beta,k}\||x|^{\beta}f_{j}\|_{L^{2}}=C_{d,\beta,k},
\end{eqnarray*}
showing that $f_{j}\in {\dot{B}_{\infty,\infty}^{-\mu-\alpha,k}}$. Hence we can apply Theorem \ref{5pth6} and use (\ref{5peq48}) to get, 
\begin{eqnarray*}
\|f_{j}\|_{\dot{B}_{\infty,\infty}^{-\mu-\alpha,k}}
\geq C_{k}>0. 
\end{eqnarray*}
In other words,
\begin{equation*}
\sup_{t>0}t^{\frac{\mu+\alpha}{2}}\|e^{t\Delta_k}f_j\|_{L^\infty}\geq C_{k}>0.
\end{equation*}
It then follows that for each $j\in\mathbb{N}$, there exists $t_j>0$ such that
\begin{equation}\label{5peq50}
t_j^{\frac{\mu+\alpha}{2}}\|e^{t_j\Delta_k}f_j\|_{L^\infty}\geq \frac{C_{k}}{2}. 
\end{equation}
Now, we define 
\begin{eqnarray*}
\Tilde{f}_j(x):=t_j^{\frac{1}{2}(\frac{d_k}{2}+\beta)}f_j(t_j^{\frac{1}{2}}x).
\end{eqnarray*}
Then it is easy to see that 
\begin{eqnarray*}
\||x|^{\beta}\tilde{f_{j}}\|_{L^{2}(\mathbb{R}^{d},w_{k})}=\||x|^{\beta}f_{j}\|_{L^{2}(\mathbb{R}^{d},w_{k})}.
\end{eqnarray*}
Also, 
\begin{eqnarray}\label{5peq47}
\||x|^{\gamma}I_{\alpha}^{k}\tilde{f_{j}}\|_{L^{s}(\mathbb{R}^{d},w_{k})}=\||x|^{\gamma}I_{\alpha}^{k} f_{j}\|_{L^{s}(\mathbb{R}^{d},w_{k})}.
\end{eqnarray}
Indeed, by Lemma \ref{5plem6} we have
\begin{eqnarray*}
I_{\alpha}^{k}\tilde{f_{j}}(x) &=& (c_{\alpha}^{k})^{-1} \int\limits_{\mathbb{R}^{d}} \tilde {f_{j}}(y)\tau_{y}^{k}|.|^{\alpha-d_{k}}(x) w_{k}(y)dy\\
&=& (c_{\alpha}^{k})^{-1}t_{j}^{\frac{1}{2}(\frac{d_{k}}{2}+\beta)}\int\limits_{\mathbb{R}^{d}} f_{j}(t_{j}^{\frac{1}{2}}y)\Phi (x,y) w_{k}(y)dy\\
&=&  (c_{\alpha}^{k})^{-1}t_{j}^{\frac{1}{2}(\frac{d_{k}}{2}+\beta)}t_{j}^{-\frac{d_{k}}{2}}\int\limits_{\mathbb{R}^{d}} f_{j}(y)\Phi (x,t_{j}^{-\frac{1}{2}}y) w_{k}(y)dy\\
&=& (c_{\alpha}^{k})^{-1}t_{j}^{\frac{1}{2}(\frac{d_{k}}{2}+\beta)}t_{j}^{-\frac{d_{k}}{2}}\int\limits_{\mathbb{R}^{d}} f_{j}(y)\Phi (t_{j}^{-\frac{1}{2}}y,x) w_{k}(y)dy\\
&=& (c_{\alpha}^{k})^{-1}t_{j}^{\frac{1}{2}(\frac{d_{k}}{2}+\beta)}t_{j}^{-\frac{d_{k}}{2}}t_{j}^{-\frac{\alpha-d_{k}}{2}}\int\limits_{\mathbb{R}^{d}} f_{j}(y)\Phi (y,t_{j}^{\frac{1}{2}}x) w_{k}(y)dy\\
&=& t_{j}^{\frac{1}{2}(\frac{d_{k}}{2}+\beta)-\frac{\alpha}{2}} I_{\alpha}^{k} f_{j}(t_{j}^{\frac{1}{2}}x).
\end{eqnarray*}
Then by substituting the value of $\gamma$ from (\ref{5peq45}), 
\begin{eqnarray*}
\||x|^{\gamma}I_{\alpha}^{k}\tilde{f_{j}}\|_{s,w_{k}}^{s}
&=& t_{j}^{\frac{s}{2}(\frac{d_{k}}{2}+\beta)-\frac{\alpha s}{2}} \int\limits_{\mathbb{R}^{d}} |I_{\alpha}^{k}f_{j}(t_{j}^{\frac{1}{2}}x)|^{s}w_{k}(x)dx\\
&=& t_{j}^{\frac{s}{2}(\frac{d_{k}}{2}+\beta)-\frac{\alpha s}{2}} t_{j}^{-\frac{\gamma s}{2}-\frac{d_{k}}{2}} \||x|^{\gamma}I_{\alpha}^{k} f_{j}\|_{s,w_{k}}^{s}\\
&=& \||x|^{\gamma}I_{\alpha}^{k} f_{j}\|_{s,w_{k}}^{s},
\end{eqnarray*}
thus proving (\ref{5peq47}). As a consequence, using (\ref{5peq48}), we have
\begin{eqnarray}\label{5peq52}
\||x|^{\beta}\tilde {f_{j}}\|_{L^{2}(\mathbb{R}^{d},w_{k})}=1~and~ \||x|^{\gamma}I_{\alpha}^{k}\tilde {f_{j}}\|_{L^{s}(\mathbb{R}^{d},w_{k})}\to W_{k},
\end{eqnarray}
which shows that $\{\tilde {f_{j}}\}_{j}$ is also a maximizing sequence for $W_{k}$.

Moreover, using (\ref{5peq4}), we observe that 
\begin{eqnarray*}
e^{1.\Delta_{k}}\tilde {f_{j}}(x)
&=&\int\limits_{\mathbb{R}^{d}} \tilde {f_{j}}(y) \tau_{y}^{k}q_{1}^{k}(x)w_{k}(y)dy\\
&=&t_j^{\frac{1}{2}(\frac{d_k}{2}+\beta)} t_{j}^{-\frac{d_{k}}{2}}\int\limits_{\mathbb{R}^{d}} f_{j}(y) \big(\tau_{t_{j}^{-\frac{1}{2}}y}^{k}q_{1}^{k}\big) (x)w_{k}(y)dy\\
&=& t_j^{\frac{1}{2}(\frac{d_k}{2}+\beta)} t_{j}^{-\frac{d_{k}}{2}}\int\limits_{\mathbb{R}^{d}} f_{j}(y) 2^{-\frac{d_{k}}{2}} e^{-\frac{|x|^{2}+|t_{j}^{-\frac{1}{2}}y|^{2}}{4}} E_{k}\bigg(\frac{x}{\sqrt{2}},\frac{t_{j}^{-\frac{1}{2}}y}{\sqrt{2}}\bigg)w_{k}(y)dy\\
&=& t_j^{\frac{1}{2}(\frac{d_k}{2}+\beta)}\int\limits_{\mathbb{R}^{d}} f_{j}(y)(2t_{j})^{-\frac{d_{k}}{2}} e^{-\frac{|t_{j}^{\frac{1}{2}}x|^{2}+|y|^{2}}{4t_{j}}}E_{k}\bigg(\frac{t_{j}^{\frac{1}{2}}x}{\sqrt{2t_{j}}},\frac{y}{\sqrt{2t_{j}}}\bigg)w_{k}(y)dy\\
&=& t_j^{\frac{1}{2}(\frac{d_k}{2}+\beta)}\int\limits_{\mathbb{R}^{d}} f_{j}(y)(\tau_{y}^{k}q_{t_{j}}^{k})(t_{j}^{\frac{1}{2}}x) w_{k}(y)dy\\
&=& t_j^{\frac{1}{2}(\frac{d_k}{2}+\beta)} (e^{t_{j}\Delta_{k}}f_{j}) (t_{j}^{\frac{1}{2}}x).
\end{eqnarray*}
Using (\ref{5peq50}), 
\begin{eqnarray}\label{5peq49}
\|e^{1.\Delta_{k}}\tilde {f_{j}}\|_{L^{\infty}}= t_j^{\frac{1}{2}(\frac{d_k}{2}+\beta)} \|(e^{t_{j}\Delta_{k}}f_{j}) (t_{j}^{\frac{1}{2}}.)\|_{L^{\infty}}
= t_{j}^{\frac{\mu+\alpha}{2}} \|e^{t_{j}\Delta_{k}}f_{j}\|\geq \frac{C_{k}}{2}>0. 
\end{eqnarray}
Since $\{\tilde{f_{j}}\}_{j}$ is a bounded sequence in $L^{2}(\mathbb{R}^{d},|x|^{2\beta} w_{k})$, by reflexivity, it has a subsequence still denoted by $\tilde{f_{j}}$ such that $\tilde{f_{j}}$ converges weakly to a function $h$ in $L^2(\mathbb{R}^d,|x|^{2\beta}w_k)$. Our aim is to show that $h$ is indeed a maximizer for $W_{k}$, that is, 
\begin{eqnarray}\label{5peq51}
\||x|^{\beta}h\|_{L^{2}(\mathbb{R}^{d},w_{k})}=1~\rm{and}~\||x|^{\gamma}I_{\alpha}^{k} h\|_{L^{s}(\mathbb{R}^{d},w_{k})}= W_{k}. 
\end{eqnarray}
Now we set
\begin{eqnarray*}
u_{j}:= I_{\alpha}^{k} \tilde{f_{j}}~\rm{and}~ v:= I_{\alpha}^{k}h.
\end{eqnarray*}
Since by Theorem \ref{5pth8}, $e^{1.\Delta_{k}}$ is a compact operator from $L^{2}(\mathbb{R}^{d},|x|^{2\beta} w_{k})$ into $L^{\infty}(\mathbb{R}^{d})$, passing through a subsequence, we have $e^{1.\Delta_{k}}\tilde{f_{j}}$ converges strongly to $e^{1.\Delta_{k}} h$ in $L^{\infty}(\mathbb{R}^{d})$. Then from (\ref{5peq49}), $\|e^{1.\Delta_{k}} h\|_{L^{\infty}}\geq\frac{C}{2}>0$, which implies that $h\not\equiv 0$. Again by taking $r=2, \beta=\gamma$ in Theorem \ref{5pth2}, we observe that all the conditions of Theorem \ref{5pth2} are satisfied under the hypothesis of the given Theorem and therefore, if $\mathcal{K}$ is a compact set in $\mathbb{R}^{d}$, we have the compact embedding
\begin{eqnarray*}
\dot{H}_{\beta,k}^{\alpha,2}(\mathbb{R}^{d})\subset L^{s}(\mathcal{K},|x|^{\beta s}w_{k}).
\end{eqnarray*}
By observing that $u_{j}$ is a bounded sequence in $\dot{H}_{\beta,k}^{\alpha,2}(\mathbb{R}^{d})$ and thereafter following the proof of Theorem \ref{5pth2}, we can show that $u_{j}$ converges strongly to $v$ in $L^{s}(\mathcal{K},|x|^{\beta s}w_{k})$. Therefore, up to a subsequence, $u_{j}$ converges to $v$ a.e. in $\mathcal{K}$ and by using diagonal argument, further, up to a subsequence, $u_{j}$ converges to $v$ a.e. in $\mathbb{R}^{d}$. 

Now proceeding exactly, as in the proof of Theorem 5.1 in \cite{pd}, we can prove that $\||x|^{\beta}h\|_{L^{2}(\mathbb{R}^{d},w_{k})}=1$ and $\tilde{f_{j}}\to h$ strongly in $L^{2}(\mathbb{R}^{d},|x|^{2\beta}w_{k})$. Since by Theorem \ref{5pth1}, the operator $I_{\alpha}^{k}$ is continuous from $L^{2}(\mathbb{R}^{d},|x|^{2\beta}w_{k})$ into $L^{s}(\mathbb{R}^{d},|x|^{\gamma s}w_{k})$, hence
\begin{eqnarray*}
u_{j}\to v ~\rm{strongly}~in~ L^{s}(\mathbb{R}^{d},|x|^{\gamma s}w_{k}).
\end{eqnarray*}
This implies that $\|u_{j}\|\to \|v\|$ in $L^{s}(\mathbb{R}^{d},|x|^{\gamma s}w_{k})$. Now (\ref{5peq51}) will follow from (\ref{5peq52}), which completes the proof.  
\end{pf}

\end{document}